\documentstyle[12pt]{article}
\newtheorem{theorem}{Theorem}
\newtheorem{lemma}[theorem]{Lemma}
\newtheorem{proposition}[theorem]{Proposition}
\newtheorem{examp}{Example}
\newtheorem{corollary}[theorem]{Corollary}
\newtheorem{remar}[theorem]{Remark}
\newtheorem{ex}{Exercise}
\newtheorem{prob}{Open Problem}
\newenvironment{proof}{Proof:\ \ \ }{\QED}
\newenvironment{example}{\begin{examp}\rm}{\diams\end{examp}}
\newenvironment{exercise}{\begingroup\baselineskip=2pt \footnotesize\rm
\begin{ex}\sl}{\end{ex}\endgroup}
\newenvironment{problem}{\begin{prob}\rm}{\end{prob}}
\newenvironment{remark}{\begin{remar}\rm}{\end{remar}}

\newcommand{\diams}{\unskip\nobreak\hfil\penalty50%
\hskip1em\hbox{}\nobreak\hfil%
$\diamondsuit$\parfillskip=0pt\finalhyphendemerits=0}
\newcommand{\QED}{{\unskip\nobreak\hfil\penalty50%
\hskip1em\hbox{}\nobreak\hfil $\Box$%
\parfillskip=0pt \finalhyphendemerits=0 \par\medskip\noindent}}
\newcommand{\bfind}[1]{\index{#1}{\bf #1}}
\newcommand{\gloss}[1]{#1\glossary{\protect #1}}
\newcommand{\n}{\par\noindent}

\newcommand{\sn}{\par\smallskip\noindent}
\newcommand{\mn}{\par\medskip\noindent}
\newcommand{\bn}{\par\bigskip\noindent}
\newcommand{\pars}{\par\smallskip}
\newcommand{\parm}{\par\medskip}

\newcommand{\isom}{\simeq}
\newcommand{\fvklit}[1]{[#1]}
\newcommand{\ind}[1]{\index{#1}{#1}}

\newcommand{\ovl}[1]{\overline{#1}}

\newcommand{\ec}{\prec_{\exists}}
\newcommand{\sep}{^{\rm sep}}
\newcommand{\chara}{\mbox{\rm char}\,}
\newcommand{\trdeg}{\mbox{\rm trdeg}\,}

\newcommand{\Gal}{\mbox{\rm Gal}\,}

\newcommand{\adresse}{\par\bigskip \small\rm
   Mathematical Sciences Group, 
   University of Saskatchewan, \par
   106 Wiggins Road, 
   Saskatoon, Saskatchewan, Canada S7N 5E6 \par
   email: fvk@math.usask.ca }
\font\tenlv=msbm10 scaled 1200
\font\sevenlv=msbm7 scaled 1200
\font\fivelv=msbm5 scaled 1200
\def\lv #1{{\mathchoice{{\hbox{\tenlv #1}}}{{\hbox{\tenlv #1}}}
{{\hbox{\sevenlv #1}}}{{\hbox{\fivelv #1}}}}}

\newcommand{\N}{\lv N}
\newcommand{\Q}{\lv Q}

\newcommand{\Z}{\lv Z}
\newcommand{\F}{\lv F}
\newcommand{\Fp}{\F_p}
\begin{document}
\title{Additive Polynomials and Their Role in the Model Theory of Valued
Fields\footnote{This paper was written while I was a guest of the Equipe
G\'eom\'etrie et Dynamique, Institut Math\'ematiques de Jussieu, Paris,
and of the Equipe Alg\`ebre--G\'eom\'etrie at the University of
Versailles. I gratefully acknowledge their hospitality and support.
I was also partially supported by a Canadian NSERC grant and by a
sabbatical grant from the University of Saskatchewan. Furthermore I am
endebted to the organizers of the conference in Teheran and the members
of the IPM and all our friends in Iran for their hospitality and
support. I also thank Peter Roquette, Florian Pop and Philip Rothmaler
for their helpful suggestions and inspiring discussions, and the two
referees for their careful reading of the paper, their corrections
and numerous useful suggestions. The final revision of this paper was
done during my stay at the Newton Institute at Cambridge; I gratefully
acknowledge its support.}}
\author{Franz-Viktor Kuhlmann\\[.4cm]
{\small\sc Dedicated to Mahmood Khoshkam ($\dagger$ October 13, 2003)}}
\date{8.\ 7.\ 2005}
\maketitle
\begin{abstract}\noindent
{\footnotesize\rm
We discuss the role of additive polynomials and $p$-polynomials in the
theory of valued fields of positive characteristic and in their
model theory. We outline the basic properties of rings of additive
polynomials and discuss properties of valued fields of positive
characteristic as modules over such rings. We prove the existence of
Frobenius-closed bases of algebraic function fields $F|K$ in one
variable and deduce that $F/K$ is a free module over the ring of
additive polynomials with coefficients in $K$. Finally, we prove that
every minimal purely wild extension of a henselian valued field is
generated by a $p$-polynomial.}
\end{abstract}
%
%
%
%
\section{Introduction}
This paper is to some extent a continuation of my introductive and
programmatic paper [Ku3]. In that paper I pointed out that the
ramification theoretical {\it defect} of finite extensions of valued
fields is responsible for the problems we have when we deal with the
model theory of valued fields, or try to prove local uniformization in
positive characteristic.

In the present paper I will discuss the connection between the defect
and additive polynomials. I will state and prove basic facts about
additive polynomials and then treat several instances where they enter
the theory of valued fields in an essential way that is particularly
interesting for model theorists and algebraic geometers. I will show
that non-commutative structures (skew polynomial rings) play an
essential role in the structure theory of valued fields in positive
characteristic. Further, I will state the main open questions.
I will also include some exercises.

In the next section, I will give an introduction to additive polynomials
and describe the reasons for their importance in the model theory of
valued fields. For the convenience of the reader, I outline the
characterizations of additive polynomials in Section~\ref{sectchar}
and the basic properties of rings of additive polynomials in
Section~\ref{sectring}. For more information on additive polynomials,
the reader may consult [Go]. The remaining sections of this paper will
then be devoted to the proofs of some of the main theorems stated in
Section~\ref{sectreas}.

%
%
\section{Reasons for the importance of additive polynomials in the model
theory of valued fields}                    \label{sectreas}
A polynomial $f\in K[X]$ is called {\bf additive}\index{additive
polynomial} if
\begin{equation}                            \label{fa+b}
f(a+b)\;=\;f(a)+f(b)
\end{equation}
for all elements $a,b$ in every extension field $L$ of $K$, that is, if
the mapping induced by $f$ on $L$ is an endomorphism of the additive group
$(L,+)$. If $K$ is infinite, then $f$ is additive already when
condition~(\ref{fa+b}) is satisfied for all $a,b\in K$: see part b) of
Corollary~\ref{charaddpol} in Section~\ref{sectchar}.

It follows from the definition that an additive polynomial cannot have a
non-zero constant term. If the characteristic $\chara K$ is zero, then
the only additive polynomials over $K$ are of the form $cX$ with $c\in
K$. If $\chara K=p>0$, then the mapping $a\mapsto a^p$ is an
endomorphism of $K$, called the \bfind{Frobenius}. Therefore, the
polynomial $X^p$ is additive over any field of characteristic $p$.
Another famous and important additive polynomial is $\wp(X):=X^p-X$, the
additive \bfind{Artin-Schreier polynomial}. An extension of a field $K$
of characteristic $p$ generated by a root of a polynomial of the form
$X^p-X-c$ with $c\in K$ is called an \bfind{Artin-Schreier extension}.
We will see later that Artin-Schreier extensions play an important role
in the theory of fields in characteristic $p$.

Note that there are polynomials defined over a finite field which
are not additive, but satisfy the condition for all elements coming from
that field. For example, we know that $a^p=a$ and thus $a^{p+1}-a^2=0$
for all $a\in \F_p\,$. Hence, the polynomial $g(X):=X^{p+1}-X^2$
satisfies $g(a+b)=0=g(a)+g(b)$ for all $a,b\in \F_p\,$. But it is not an
additive polynomial. To show this, let us take an element $\vartheta$ in
the algebraic closure of $\F_p$ such that $\vartheta^p-\vartheta=1$.
Then $g(\vartheta)=\vartheta(\vartheta^p-\vartheta)=\vartheta$. On the
other hand, $g(\vartheta+1)=(\vartheta+1)((\vartheta+1)^p-(\vartheta
+1))=(\vartheta+1)(\vartheta^p+1^p-\vartheta-1)=\vartheta+1\ne\vartheta=
g(\vartheta)+g(1)$. Hence, already on the extension field
$\F_p(\vartheta)$, the polynomial $g$ does not satisfy the additivity
condition.

The following well known theorem gives a very useful characterization of
additive polynomials. I will present a proof in Section~\ref{sectchar}.

\begin{theorem}                             \label{thmaddpol}
Let $p$ be the characteristic exponent of $K$ (i.e., $p=\chara K$ if
this is positive, and $p=1$ otherwise). Take $f\in K[X]$. Then
$f$ is additive if and only if it is of the form
\begin{equation}                            \label{formaddpol}
f(X)\;=\;\sum_{i=0}^{m}c_i X^{p^i}\;\;\;\mbox{ with } c_i\in K\;.
\end{equation}
\end{theorem}

Assume that $\chara K =p>0$. Then as a mapping on $K$, $X^p$ is equal to
the Frobenius endomorphism $\varphi$. Similarly, $X^{p^2}$ is equal to
the composition of $\varphi$ with itself, written as $\varphi^2$, and by
induction, we can replace $2$ by every integer $n$. On the other hand,
the monomial $X$ induces the identity mapping, which we may write as
$\varphi^0$. Note that addition and composition of additive mappings on
$(K,+)$ give again additive mappings (in particular, addition of
additive polynomials gives additive polynomials). It remains
to interpret the coefficients of additive polynomials as mappings. This is
easily done by viewing $K$ as a $K$-vector space: the mapping $c\cdot$
induced by $c\in K$ is given by multiplication $a\mapsto ca$, and it is
an automorphism of $(K,+)$ if $c\not=0$.
So $cX^{p^n}$ as a mapping is the composition of
$\varphi^n$ with $c\cdot\,$. We will write this composition as
$c\varphi^n$. Adding these monomials generates new additive mappings of
the form $\sum_{i=0}^{m}c_i \varphi^i$, and addition of such mappings
gives again additive mappings of this form. Composition of such additive
mappings generates again additive mappings, and the reader may compute that
they can again be written in the above form. In this way, we are
naturally led to considering the ring \gloss{$K[\varphi]$} of all
polynomials in $\varphi$ over $K$, where multiplication is given by
composition. From the above we see that this ring is a subring of the
endomorphism ring of the additive group of $K$. The correspondence that
we have worked out now reads as
\begin{equation}                            \label{corraddpol}
\sum_{i=0}^{m}c_i X^{p^i} \;\;\;\longleftrightarrow\;\;\;
\sum_{i=0}^{m}c_i \varphi^i \in K[\varphi]
\end{equation}
which means that both expressions describe the same additive mapping on $K$.
For instance, the additive Artin-Schreier polynomial $X^p -X$
corresponds to $\varphi -1\,$. Through the above correspondence, the
ring $K[\varphi]$ may be considered as the \bfind{ring of additive
polynomials over $K$}. Note that this ring is not commutative; in fact,
we have
\[\varphi c = c^p\varphi\;\;\; \mbox{ for all } c\in K\;.\]
This shows that assigning $\varphi\mapsto z$ induces an isomorphism of
$K[\varphi]$ onto the skew polynomial ring $K[z;\varphi]$. But we will
keep the notation ``$K[\varphi]$'' since it is simpler.

\pars
Let me state some basic properties of the ring $K[\varphi]$. For more
information, I recommend the comprehensive book ``Free rings and their
relations'' by \ind{P.~M.~Cohn} (\fvklit{C1}, \fvklit{C2}). Let $R$ be a
ring with $1\ne 0$. Equipped with a function $\mbox{ deg} :\> R
\setminus\{0\}\rightarrow
\N\cup\{0\}$, the ring $R$ is called \bfind{left euclidean} if for all
elements $s,s'\in R$, $s\not=0$, there exist $q,r\in R$ such that
\[
s'=qs + r\;\;\;\mbox{ with }\;\; r=0\;\mbox{ or }\;\deg r <\deg s\;,
\]
and it is called \bfind{right euclidean} if the same holds with
``$s'=sq + r$'' in the place of ``$s'=qs + r$''. (Usually, the function
deg is extended to $0$ by setting $\deg 0=-\infty$.) For example,
polynomial rings over fields equipped with the usual degree function are
both-sided euclidean rings. Further, an integral domain $R$ is called a
\bfind{left principal ideal domain} if every left ideal in $R$ is
principal (and analogously for ``right'' in the place of
``left''). I leave it to the reader to show that every left (or
right) euclidean integral domain is a left (or right) principal
ideal domain. Finally, an integral domain $R$ is called a \bfind{left
Ore domain} if
\[Rr\cap Rs\not=\{0\}\;\;\;\mbox{ for all $r,s\in R\setminus
\{0\}$}\;,\]
and it is called a \bfind{right Ore domain} if $rR\cap sR\not=\{0\}$
for all $r,s\in R\setminus \{0\}$. Every left (or right) Ore domain
can be embedded into a skew field (cf.\ \fvklit{C1}, \S 0.8,
Corollary 8.7). The reader may prove that a left (or right)
principal ideal domain is a left (or right) Ore domain.

The ring $K[\varphi]$ may be equipped with a degree function which
satisfies $\deg 0 =-\infty$ and $\deg \sum_{i=0}^{m}c_i \varphi^i = m$
if $c_m\not=0$. This degree function is a homomorphism of the
multiplicative monoid of $K[\varphi]\setminus\{0\}$ onto $\N\cup\{0\}$
since it satisfies $\deg rs = \deg r +\deg s$. In particular, this shows
that $K[\varphi]$ is an integral domain. The following theorem is due to
\ind{O.~Ore} \fvklit{O2}; I will give a proof in Section~\ref{sectring}.

\begin{theorem}                             \label{Kphi}
The ring $K[\varphi]$ is a left euclidean integral domain and thus also
a left principal ideal domain and a left Ore domain. It is right euclidean
if and only if $K$ is perfect; if $K$ is not perfect, then $K[\varphi]$ is
not even right Ore.
\end{theorem}

\pars
\begin{example}
Let $\F_p$ denote the field with $p$ elements. The ring $\Fp[\varphi]$
is a both-sided euclidean integral domain, and every field $K$ of
characteristic $p$ is a left $\Fp[\varphi]$-module and a left
$K[\varphi]$-module, where the action of $\varphi$ on $K$ is just the
application of the Frobenius endomorphism. $K$ is perfect if and only if
every element of $K$ is divisible by the ring element $\varphi$. But
this does not imply that $K$ is a divisible $\Fp[\varphi]$- or
$K[\varphi]$-module. For instance, if $K$ admits non-trivial
Artin-Schreier extensions, that is, if $K\not= \wp(K) = (\varphi-1)K$,
then there are elements in $K$ which are not divisible by $\varphi -1$.
On the other hand, $K$ is a divisible $\Fp[\varphi]$- and
$K[\varphi]$-module if $K$ is algebraically closed.

Observe that $K$ is not torsion free as an $\Fp[\varphi]$- or
$K[\varphi]$-module. Indeed, $K$ contains $\Fp$ which satisfies
\[(\varphi-1)\Fp=\{0\}\;.\]
\end{example}

\begin{example}
The power series field $K:=\F_p((t))=\{\sum_{i=N}^{\infty}c_i t^i\mid
N\in \Z\,,\,c_i\in K\}$ (also called ``field of formal Laurent series
over $\Fp$'') is not perfect, since $t$ does not admit a $p$-th root in
$K$. Hence, the ring $K[\varphi]$ is not right Ore. $K$ is a left
$K[\varphi]$-module.
\end{example}

In Section~\ref{sectring}, Remark~\ref{remCohn}, I will collect
a few properties of the rings $K[\varphi]$ that follow from
Theorem~\ref{Kphi}, and describe what happens if $K$ is not perfect. We
will see that in that case the structure of $K[\varphi]$-modules becomes
complicated. It seems that the ``bad'' properties of $K[\varphi]$, for
$K$ non-perfect, are symptomatic for the problems algebraists and model
theorists have with non-perfect valued fields in positive
characteristic. Let me discuss the most prominent of such non-perfect
valued fields.

The field $\F_p((t))$ carries a canonical valuation $v_t\,$, called the
$t$-adic valuation. It is given by $v_t \sum_{i=N}^{\infty}c_i t^i =
N$ if $c_N\ne 0$ and $v_t 0=\infty$.
$(\F_p((t)),v_t)$ is a complete discretely valued field, with value
group $v_t\F_p((t))=\Z$ (that is what ``discretely valued'' means) and
residue field $\F_p((t))v_t =\F_p\,$. At the first glance, such fields
may appear to be the best known objects in valuation theory.
Nevertheless, the following prominent questions about the elementary
theory Th$(\F_p((t)),v_t)$ are still unanswered:

\begin{problem}
Is the elementary theory of the valued field $\F_p((t))$ model complete?
Does it admit quantifier elimination in some natural language? Is it
decidable? If yes, what would be a complete recursive axiomatization?
\end{problem}

The corresponding problem for the $p$-adics was solved in the mid 1960s
independently by Ax and Kochen [A--K] and by Ershov [E]. Since then, the
above problem has been well known to model-theoretic algebraists, but
resisted all their attacks.

\parm
Encouraged by the similarities between $\F_p((t))$ and the field $\Q_p$
of $p$-adic numbers, one might try to give a complete axiomatization for
$\mbox{\rm Th}(\F_p((t)),v_t)$ by adapting the well known axioms for
$\mbox{\rm Th}(\Q_p,v_p)$. They express that $(\Q_p,v_p)$ has the
following elementary properties:
\sn
$\bullet$ \ It is a henselian valued field of characteristic 0. A valued
field $(K,v)$ is called henselian if it satisfies Hensel's Lemma: {\it
If $f$ is a polynomial with coefficients in the valuation ring ${\cal
O}$ of $v$ and if $b\in {\cal O}$ such that $vf(b)>0$ while $vf'(b)=0$,
then there is some $a\in {\cal O}$ such that $f(a)=0$ and $v(a-b)>0$.}
This holds if and only if the extension of $v$ to the algebraic closure
of $K$ is unique.
\sn
$\bullet$ \ Its value group a $\Z$-group, i.e., an ordered abelian group
elementarily equivalent to $\Z$.
\sn
$\bullet$ \ Its residue field is $\F_p\,$.
\sn
$\bullet$ \ $v_p p$ is equal to $1$ ($=$ the smallest
positive element in the value group).
\sn
The last condition is not relevant for $\F_p((t))$ since there, $p\cdot
1=0$. Nevertheless, we may add a constant name $t$ to ${\cal L}$ so that
one can express by an elementary sentence that $v_t t=1$.

A naive adaptation would just replace ``characteristic 0'' by
``characteristic $p$'' and $p$ by $t$. But
there is an elementary property of valued fields that is satisfied by
all valued fields of residue characteristic 0 and all formally $p$-adic
fields, but not by all valued fields in general. It is the property of
being defectless. A valued field $(K,v)$ is called {\bf defectless} if
for every finite extension $L|K$, equality holds in
the {\bf fundamental inequality}
\begin{equation}                            \label{fundineq}
n\;\geq\;\sum_{i=1}^{\rm g} {\rm e}_i {\rm f}_i\;,
\end{equation}
where $n=[L:K]$ is the degree of the extension, $v_1,\ldots,v_{\rm g}$
are the distinct extensions of $v$ from $K$ to $L$, ${\rm e}_i=
(v_iL:vK)$ are the respective ramification indices, and ${\rm f}_i=
[Lv_i:Kv]$ are the respective inertia degrees. (Note that ${\rm g}=1$ if
$(K,v)$ is henselian.) There is a simple example, due to F.~K.~Schmidt,
which shows that there is a henselian discretely valued field of
positive characteristic which is not defectless (cf.\ [Ri], Exemple 1,
p.\ 244). This field has a finite purely inseparable extension with
non-trivial defect. But defect does not only appear in purely inseparable
extensions: there is an example, due to A.\ Ostrowski, of a complete
valued field admitting a finite separable extension with non-trivial
defect (cf.\ [Ri], Exemple 2, p.\ 246). These and several other examples
of extensions with non-trivial defect of various types can also be found
in [Ku12] (see also [Ku8]).

However, each power series field with its canonical valuation is
henselian and defectless (cf.\ [Ku12]). In particular, $(\F_p((t)),v_t)$
is defectless. For a less naive adaptation of the axiom system of
$\Q_p$, we will thus add ``defectless''. We obtain the following axiom
system in the language ${\cal L}(t)={\cal L}\cup\{t\}$:
\begin{equation}                            \label{AS1}
\left.\begin{array}{l}
(K,v)\mbox{ is a henselian defectless valued field}\\
K\mbox{ is of characteristic }p\\
vK\mbox{ is a $\Z$-group}\\
Kv=\F_p\\
vt\mbox{ is the smallest positive element in }vK\;.
\end{array}\;\;\;\right\}
\end{equation}
Let us note that also $(\F_p(t),v_t)^h$, the henselization of $(\F_p(t),
v_t)$, satisfies these axioms. The \bfind{henselization} of a valued
field $(K,v)$ is a henselian algebraic extension which is minimal in the
sense that it admits a (unique!) embedding over $K$ in every henselian
extension of $(K,v)$. Henselizations exist for all valued fields, and
they are separable extensions (cf.\ [Ri], Th\'eor\`eme 2, p. 176). It is
well known that $(\F_p(t),v_t)^h$ is a defectless field, being the
henselization of a global field (cf.\ [Ku9]). It is also well known that
$(\F_p(t),v_t)^h$ is existentially closed in $(\F_p((t)),v_t)$ (see
below for the definition of this notion); this fact follows from work of
Greenberg [Gre] and also from Theorem 2 of [Er1] (see also [Ku7] for
some related information). But it is not known whether $(\F_p((t)),v_t)$
is an elementary extension of $(\F_p(t),v_t)^h$.

It did not seem unlikely that axiom system (\ref{AS1})
could be complete, until I proved in [Ku1] (cf.\ [Ku4]):
\begin{theorem}                             \label{notc}
The axiom system (\ref{AS1}) is not complete.
\end{theorem}

I will give an idea of the proof of this theorem in
Section~\ref{sectepap} below. It was inspired by an observation of Lou
van den Dries. He had worked with a modified axiom system (with larger
residue fields) and had found an elementary sentence which he was not
able to deduce from that axiom system (as it turned out, that wasn't
van den Dries' fault). This sentence was formulated using
only addition, multiplication with the element $t$ and application of
the Frobenius, but no general multiplication. This led van den Dries to
the question whether one could at least determine the model theory of
$\F_p((t))$ as a module which admits multiplication with $t$ and
application of the Frobenius, forgetting about general multiplication.
But this means that we view $\F_p((t))$ as a left $K[\varphi]$-module,
where the field $K$ contains $t$ and should be contained in $\F_p((t))$.
But then, $K$ is not perfect, and therefore the structure of
$K[\varphi]$-modules may be quite complicated.

There is a common feeling that additive polynomials play a crucial role
in the theory of valued fields of positive characteristic. So indeed,
van den Dries' question may be the key to the model theory of
$\F_p((t))$ (but it could be as hard to solve as the original problem).
In this paper, I will give some reasons for this common feeling, but
also confront it with our present problem of understanding the notion
of extremality.

%
%
\subsection{Reason \#1: Kaplansky's hypothesis A}
For a valued field $(K,v)$, we denote by $vK$ its value group and by
$Kv$ its residue field. An extension $(K,v)\subset (L,v)$ of valued
fields is called \bfind{immediate} if the induced embeddings of $vK$ in
$vL$ and of $Kv$ in $Lv$ are onto. Henselizations are immediate
extensions (cf.\ [Ri], Corollaire 1, p.\ 184). Wolfgang Krull [Kr] (see
also [Gra]) proved that every valued field admits a maximal immediate
extension. A natural and in fact very important question is whether this
is unique up to (valuation preserving) isomorphism. This plays a role
when one wishes to embed valued fields in power series fields. In his
celebrated paper [Ka1], Irving Kaplansky gave a criterion, called
``hypothesis A'', which guarantees uniqueness. (We will present it
later.) Kaplansky then showed that a valued field $(K,v)$ of positive
characteristic having a cross-section and satisfying hypothesis A
can be embedded in the power series field $Kv((vK))$ over its residue
field $Kv$ with exponents in its value group $vK$. Kaplansky also gives
examples which show that this may fail if hypothesis A is not satisfied.
In this case, there are \bfind{maximal fields} (= valued fields not
admitting any proper immediate extensions) which do not have the form of
a power series field (not even if one allows factor systems).

If we are considering an elementary class of valued fields satisfying
hypothesis A (which can be expressed by a recursive scheme of
elementary sentences in the language of valued rings), then the
uniqueness of maximal immediate extensions can be fruitfully used in the
proof of model theoretic properties. Let us give the example of
algebraically maximal Kaplansky fields. A valued field is called
\bfind{algebraically maximal} if it does not admit any proper immediate
algebraic extension. It is called a \bfind{Kaplansky field} if it
satisfies hypothesis A. The following theorem is due to Ershov [Er1]
and, independently, Ziegler [Zi].
\begin{theorem}
The elementary theory of an algebraically maximal Kaplansky field is
completely determined by the elementary theory of its value group and
the elementary theory of its residue field.
\end{theorem}

In other words, algebraically maximal Kaplansky fields satisfy the
following \bfind{Ax--Kochen--Ershov principle}:
\begin{equation}                            \label{AKEeq}
vK\>\equiv\>vL\;\wedge\; Kv\>\equiv\>Lv\;\;\Longrightarrow\;\;
(K,v)\>\equiv\>(L,v)
\end{equation}
where the first elementary equivalence is in the language of ordered
groups, the second in the language of rings and the third in the
language of valued rings. In the case of $(K,v)\subseteq (L,v)$ there is
also a version of the Ax--Kochen--Ershov principle with ``$\equiv$''
replaced by ``$\prec$'' (elementary extension). In the same
situation, there is also the more basic version
\begin{equation}                            \label{AKEec}
vK\>\ec\>vL\;\wedge\; Kv\>\ec\>Lv\;\;\Longrightarrow\;\;
(K,v)\>\ec\>(L,v)
\end{equation}
where ``$\ec$'' means ``\bfind{existentially closed in}'', that is,
every existential elementary sentence which holds in the upper structure
also holds in the lower structure. In fact, it has turned out that
proving this version is the essential step in proving Ax--Kochen--Ershov
principles and other model theoretic results about valued fields; the
further results then often follow by general model theoretic arguments
(the reader should think of Robinson's Test).

Hypothesis A implicitly talks about additive polynomials. Following
Kaplansky [Ka2], we will call a polynomial $f\in K[X]$ a
\bfind{$p$-polynomial} if it is of the form
\begin{equation}
f(X)\;=\; {\cal A}(X)+c\,,
\end{equation}
where ${\cal A}\in K[X]$ is an additive polynomial, and $c$ is a
constant in $K$. A field $K$ of characteristic $p>0$ will be called
\bfind{$p$-closed} if every $p$-polynomial in $K[X]$ has a root in $K$.
That is,
\begin{center}
\it $K$ is $p$-closed if and only if it is a divisible
$K[\varphi]$-module.
\end{center}
In particular, every $p$-closed field is perfect.

\pars
Now hypothesis A for a valued field $(K,v)$ with $\chara Kv=p>0$ reads
as follows:
\sn
{\bf (A1)} \ the value group $vK$ is $p$-divisible, and\n
{\bf (A2)} \ the residue field $Kv$ is $p$-closed.
\sn
For valued fields $(K,v)$ with $\chara Kv=0$, hypothesis A is empty. The
condition of a field to be $p$-closed seemed obscure at the time of
Kaplansky's paper. But we have learned to understand this condition
better. The following theorem was first proved by Whaples in
\fvklit{Wh2}, using the cohomology theory of additive polynomials. A
more elementary proof was later given in \fvklit{Del}. Then Kaplansky
gave a short and elegant proof in his ``Afterthought: Maximal Fields
with Valuation'' ([Ka2]). We will reproduce this proof in
Section~\ref{sectp-cl}.

\begin{theorem}                             \label{Whaples}
A field $K$ of characteristic $p>0$ is $p$-closed if and only if it does
not admit any finite extensions of degree divisible by $p$.
\end{theorem}

This theorem lets us understand hypothesis A much better. Based on this
insight, Kaplansky's result about uniqueness of maximal immediate
extensions is reproved in [Ku--Pa--Ro]. There, it is deduced from the
Schur--Zassenhaus Theorem about conjugacy of complements in profinite
groups, via Galois correspondence.

As we are shifting our focus to additive polynomials, the original
condition ``$p$-closed'' regains its independent interest. In
Section~\ref{sectp-cl} we will use Theorem~\ref{Whaples} to prove:

\begin{theorem}                             \label{amp-cl=Kap}
A henselian valued field of characteristic $p>0$ is $p$-closed if and
only if it is an algebraically maximal Kaplansky field.
\end{theorem}
\n
For a generalization of the notion ``$p$-closed'' and of this theorem
to fields of characteristic $0$ see [V], in particular Corollary~5.

\pars
By Theorem~\ref{Whaples} we can split condition (A2) into two distinct
conditions:\sn
{\bf (A2.1)} \ the residue field $Kv$ is perfect, and\n
{\bf (A2.2)} \ the residue field $Kv$ does not admit any finite
separable extension of degree divisible by $p$.
\sn
While (A2.1) is a perfectly natural condition about fields, (A2.2) is
somewhat unusual. This is the reason for the fact that Kaplansky fields
are not often found in applications. Certainly also the other conditions
restrict the possible applications (for example, $\F_p((t))$ doesn't
satisfy (A1)). But for instance, every perfect valued field of
characteristic $p>0$ satisfies conditions (A1) and (A2.1) (but not
necessarily (A2.2)). So we would obtain a more natural condition if we
could drop condition (A2.2). To obtain good model theoretic properties
for fields satisfying (A1) and (A2.1), one has to require again that
they are algebraically maximal. Such fields form a part of an important
larger class of valued fields, the tame fields. A \bfind{tame field} is
a henselian field whose algebraic closure is equal to the ramification
field $K^r$ of the normal extension $K\sep|K$, where $K\sep$ denotes the
separable-algebraic closure of $K$. The \bfind{ramification field} of a
normal separable-algebraic extension of valued fields is the fixed
field in that extension of a certain subgroup of the Galois group, the
\bfind{ramification group}. We don't need the definition of this group
here; we only need the basic properties of the field $K^r$ which I will
put together in Theorem~\ref{ram} below. By part e) of this theorem,
every tame field is defectless, and it follows directly from the
definition that every tame field is perfect. In [Ku1] (cf.\ also [Ku11])
I proved:

\begin{theorem}                             \label{mttame}
The elementary theory of a tame field is completely determined by the
elementary theory of its value group and the elementary theory of its
residue field.
\end{theorem}

All tame fields satisfy conditions (A1) and (A2.1), but not necessarily
(A2.2). That means that we have lost the uniqueness of maximal immediate
extensions. But the above result shows that uniqueness is not necessary
for an elementary class of valued fields to have good model theoretic
properties. However, we have to work much harder. This work is again
directly related to additive polynomials, and we will describe this
connection now.

%
%
\subsection{Reason \#2: the defect and purely wild extensions}
Let us assume that $(K,v)$ is henselian. Then for every finite
extension $L$ of $K$, we have $g=1$ in the fundamental inequality
(\ref{fundineq}). Then the Lemma of Ostrowski (cf.\ [Ri], Th\'eor\`eme
2, p.\ 236) tells us that we have an equality
\begin{equation}                            \label{fundeq}
[L:K]\;=\;(vL:vK)\cdot [Lv:Kv]\cdot p^\delta\;,
\end{equation}
where $p$ is the characteristic exponent of the residue field $Kv$, and
$\delta$ is a non-negative integer. The factor $p^\delta$ is called the
\bfind{defect} of the extension $(L|K,v)$; it is \bfind{trivial} if
$p^\delta=1$. Consequently, every valued field with residue field of
characteristic $0$ is defectless.

It follows from Theorem~\ref{ram} that a valued field is tame if it is
henselian and for every finite extension $L|K$,
\sn
$\bullet$ \ the characteristic of $Kv$ does not divide $(vL:vK)$,
\sn
$\bullet$ \ the extension $Lv|Kv$ is separable, and
\sn
$\bullet$ \ the extension $(L|K,v)$ is defectless.
\sn
The ramification field $K^r$ is the unique maximal tame extension of
every henselian field $(K,v)$.

\pars
As I have explained in [Ku3], the presence of non-trivial defect is one
of the main obstacles in proving an Ax--Kochen--Ershov principle like
(\ref{AKEec}). Let me quickly sketch this again. Assume that $(L,v)$ is
an extension of a henselian field $(K,v)$ such that $vK\ec vL$ and
$Kv\ec Lv$. Then we take $(K^*,v^*)$ to be an $|L|^+$-saturated
elementary extension of $(K,v)$. It follows that $v^*K^*$ is a
$|vL|^+$-saturated extension of $vK$; hence $vK\ec vL$ yields that $vL$
can be embedded over $vK$ in $v^*K^*$. It also follows that $K^*v^*$ is
an $|Lv|^+$-saturated extension of $vL$; hence $Kv\ec Lv$ yields that
$Lv$ can be embedded over $Kv$ in $K^*v^*$. Now we have to lift these
embeddings to an embedding of $(L,v)$ in $(K^*,v^*)$ over $K$. Once this
is achieved, we are done, because every existential elementary sentence
which holds in $(L,v)$ carries over to its image in $(K^*,v^*)$, from
there up to $(K^*,v^*)$, and from there down to the elementary
substructure $(K,v)$.

By a general model theoretic argument, the situation can be reduced to
the case where $L$ is finitely generated over $K$. That is, $(L|K,v)$ is
a valued function field (by ``function field'', we will always mean
``algebraic function field''). Hence, we need the structure theory of
valued function fields to solve our embedding problem (as it is the case
for the problem of local uniformization). Let us assume that we can
reduce to the case where the transcendence degree of $L|K$ is 1. This
can be done for tame fields, but for the model theory of $\F_p((t))$,
this is another serious problem, again connected with additive
polynomials (see Section~\ref{sectepap}). Assume further that we have
lifted the embeddings of $vL$ and $Lv$ to an embedding of some subfield
$L_0$ of $L$. Then $L|L_0$ is a finite immediate extension, and in
general, it will be proper (i.e., $L\ne L_0$). Taking henselizations, we
obtain that also $L^h|L_0^h$ is a finite immediate extension. Since we
assumed that $(K,v)$ is henselian (which is true for every
algebraically maximal and every tame field), its elementary
extension $(K^*,v^*)$ is also henselian. Therefore, the embedding of
$L_0$ extends to an embedding of $L_0^h$ in $K^*$ (this is the universal
property of the henselization). But if $L^h\ne L_0^h$, we do not know
how to lift the extension further (which we would need to get all of
$L$ embedded), unless we have uniqueness of maximal immediate
extensions. Since $L^h|L_0^h$ is immediate, we have $(vL^h: vL_0^h)=1$
and $[L^hv:L_0^hv]=1$; hence if $L^h\ne L_0^h$, then by (\ref{fundeq}),
the extension has non-trivial defect, equal to its degree.

We see that indeed, the presence of non-trivial defect constitutes a
serious obstacle for our embedding problem. So we have to avoid the
defect. In certain cases, it will not even appear. All tame fields are
defectless fields (and so are all other valued fields for which we know
good model theoretic results). This does not mean that every valued
function field over a tame field is defectless. But for a certain type
of valued function fields, this is true.
Let $(L|K,v)$ be an extension of valued fields of finite transcendence
degree. Then the following inequality holds (cf.\ [B], Chapter VI,
\S10.3, Theorem~1):

\begin{equation}                            \label{wtdgeq}
\trdeg L|K \>\geq\> \dim_\Q (\Q\otimes (vL/vK) ) \,+\, \trdeg Lv|Kv\;.
\end{equation}
If equality holds then we will say that \bfind{$(L|K,v)$ is
without transcendence defect}. For such function fields, we have
([Ku1], [Ku9]):

\begin{theorem}                           \label{ai}
{\bf (Generalized Stability Theorem)} \
Let $(L|K,v)$ be a valued function field without transcendence defect.
If $(K,v)$ is a defectless field, then also $(L,v)$ is a defectless
field.
\end{theorem}

Using this theorem, one can prove (cf.\ [Ku9]):
\begin{theorem}
Let $(K,v)$ be a henselian defectless field. Then the Ax--Kochen--Ershov
principle (\ref{AKEec}) holds for every extension $(L,v)$ of $(K,v)$
without transcendence defect.
\end{theorem}

I proved Theorem~\ref{ai} in [Ku1]. How does this proof work? How can we
see whether a given valued field is defectless? First of all, a valued
field is defectless if and only if its henselization is (see [Ku9]; a
partial proof is also given in [En]). So we can assume that $L$ is the
henselization of a valued function field. Second, if $(k,v)$ is any
henselian field, then every finite extension of $k$ inside the
ramification field $k^r$ has trivial defect, and if $k_1|k$ is any
finite extension, then $k_1|k$ and $k^r.k_1|k^r$ have the same defect
(cf.\ Theorem~\ref{ram}). So in our situation, we have to show that
every finite extension of $L^r$ has trivial defect. The advantage of
working over $L^r$ is that general ramification theory tells us that
$L\sep|L^r$ is a $p$-extension. A normal and separable field extension
is called a \bfind{$p$-extension} if its Galois group is a
pro-$p$-group. It follows from the general theory of $p$-groups (cf.\
[H], Chapter III, \S 7, Satz 7.2 and the following remark) via Galois
correspondence that every finite separable-algebraic extension of $L^r$
is a tower of Galois extensions of degree $p$. Hence we just have to
show by induction that each of them has trivial defect. (The
complementary case of purely inseparable extensions is much easier and
can be disposed of more directly.) Now every Galois extension $k'|k$ of
degree $p$ of fields of characteristic $p$ is an Artin-Schreier
extension; this well known fact is proved by an application of Hilbert's
Theorem 90. We include a proof in Section~\ref{sectagfe}
(Theorem~\ref{GalpAS}), as a special case of a generalization which we
will discuss below.

Let $\vartheta$ be a root of $X^p-X-a$. Then $k'=k(\vartheta)=
k(\vartheta-c)$ for every $c\in k$. As $X^p-X$ is an additive
polynomial, we have $(\vartheta-c)^p-(\vartheta-c)=a-c^p+c$, that is,
$\vartheta-c$ is a root of the $p$-polynomial $X^p-X-(a-c^p+c)$.
So we may change $a$ by subtracting elements in $k$ of the form
$c^p-c$, without changing the extension generated by the
polynomial. The idea in our above situation is now to find by this
method a suitable normal form for the element $a$ from which we can read
off that the extension has trivial defect. The idea of deducing suitable
normal forms for Artin-Schreier extensions (and for Kummer extensions in
the case of fields of characteristic $0$) can already be found in the
work of Hasse, Whaples, Epp ([Ep], see also [Ku5]), Matignon and
Abhyankar.

Let us quickly discuss two examples. We wish to show that a given
Artin-Schreier extension $L'|L^r$ has trivial defect. Before we go on,
we note that by valuation theoretical arguments, the proof of
Theorem~\ref{ai} can be reduced to the case where $K$ and hence
also its residue field $Kv$ is algebraically closed. Assume that the
transcendence degree of $L|K$ is 1. Then by (\ref{wtdgeq}) with
equality, we can have
\sn
$\bullet \; \dim_\Q (\Q\otimes (vL/vK))=1$ and $\trdeg Lv|Kv=0$, or\n
$\bullet \; \dim_\Q (\Q\otimes (vL/vK))=0$ and $\trdeg Lv|Kv=1$.
\sn
In the first case, there is an element $x\in L$ such that $vx$ is
rationally independent over $vK$.
Under certain additional conditions, we can then take $a$ to be a
polynomial in $x$ with coefficients in $K$. Since the values of the
monomials in this polynomial $a=a(x)$ lie in distinct cosets modulo
$vK$, their values are distinct. By the ultrametric triangle law, this
implies that the value of such a monomial is equal to the least value of
its monomials. Now we can use the above method to get rid of $p$-th
powers of $x$ in $a(x)$ (we can replace a monomial $cx^{kp}$ by
$c^{1/p}x^k$. Therefore, we can assume that all monomials appearing in
the polynomial $a(x)$ are of the form $c_ix^i$ with $i$ not divisible by
$p$. Then the value $va(x)$ is not divisible by $p$ in $vL^r$. This
value cannot be positive since otherwise, the extension would be trivial
by Hensel's Lemma. With the value being negative, the reader may show
that if $\vartheta$ is a root of $X^p-X-a(x)$, then
\[v\vartheta\;=\;\frac{va(x)}{p}\;.\]
This implies that $(vL':vL^r)=p=[L':L^r]$, so the extension has trivial
defect.

In the second case, we will have an element $x\in L$ of value $0$ whose
residue $xv$ is transcendental over $Kv$. Now we will have to deal with
finite sums of the form $c_id_i$ where $d_i\in L$ are representatives of
elements in the residue field. We play the same game as before, trying
to come up with a residue that has no $p$-th root, from which it would
follow in a similar way as above that $[L'v:L^rv]=p=[L':L^r]$, showing
that the extension has trivial defect. The problem here is that when we
replace some monomial $c_id_i$ by its $p$-th root $c_i^{1/p}d_i^{1/p}$,
then even if the residue $d_i^{1/p}v$ does not have a $p$-th root in
$Lv$, the element $d_i^{1/p}$ might sum up with some other $d_j$ to
an element whose residue has again a $p$-th root in $Lv$. We somehow
have to see that this process cannot go on infinitely. A good idea would
be to take the $d_i$ such that their residues form a basis of $Lv|Kv$.
But then we would need that also the residue of $d_i^{1/p}$ is in this
basis. It is easily seen that a basis being closed under taking $p$-th
roots (as long as we stay in $Lv$) is the same as a basis being closed
under taking $p$-th powers (in other words, being closed under the
Frobenius). Such a basis will be called a \bfind{Frobenius-closed
basis}. See Lemma~\ref{Frpp} which gives the exact formulation of the
property of a Frobenius-closed basis that we need in [Ku9].

The residue field of $K(x)$ is just $Kv(xv)$. Further, $L$ being a
function field of transcendence degree 1, $L|K(x)$ is a finite
extension. It follows from the fundamental inequality that also
$Lv|K(x)v$ is a finite extension. This shows that $Lv|Kv$ is a function
field of transcendence degree 1. So in order to prove our theorem in the
second case, our task is to find a Frobenius-closed basis for every
function field of transcendence degree 1 over an algebraically closed
field of positive characteristic. In [Ku1], I proved a more general
result:

\begin{theorem}                       \label{Fc}
Let $F$ be an algebraic function field of transcendence degree $1$ over
a perfect field $K$ of characteristic $p>0$. If $K$ is relatively
algebraically closed in $F$, then there exists a Frobenius-closed basis
for $F|K$.
\end{theorem}

The proof and some further background are given in
Section~\ref{sectFrcb}. There, we will also deduce the following result
from Theorem~\ref{Fc}, showing the connection between Frobenius-closed
bases and additive polynomials:

\begin{theorem}                             \label{freeKfmod}
If $F$ is an algebraic function field of transcendence degree $1$ over a
perfect field $K$ of characteristic $p>0$ and if $K$ is relatively
algebraically closed in $F$, then $F/K$ is a free $K[\varphi]$-module.
\end{theorem}

\parm
The second important theorem that I use in the proof of
Theorem~\ref{mttame} is needed when the valued function field $(L|K,v)$
has non-trivial transcendence defect, i.e., equality does not hold in
(\ref{wtdgeq}). In reducing to transcendence degree $1$ by induction,
one reaches the case where $(L,v)$ is an immediate extension of
transcendence degree $1$ of the tame field $(K,v)$. The defect is then
avoided by means of the following theorem.

\begin{theorem}                \label{HR}
{\bf (Henselian Rationality)}\n
Let $(K,v)$ be a tame field and $(L|K,v)$ an immediate function field of
trans\-cendence degree $1$. Then the henselization $(L,v)^h$ of $(L,v)$
is \bfind{henselian rational}, i.e.,
\begin{equation}
\mbox{there is $x\in L$ such that }\; L^h\,=\,K(x)^h\;.
\end{equation}
\end{theorem}
For valued fields of residue characteristic 0, the assertion is a direct
consequence of the fact that every such field is defectless (in fact,
every $x\in L\setminus K$ will then do the job). In contrast to this,
the case of positive residue characteristic requires a much deeper
structure theory of immediate algebraic extensions of henselian fields,
in order to find suitable elements $x$. I proved this theorem in [Ku1]
(cf.\ also [Ku10]).

The proof works as follows. Suppose we have chosen the wrong $x$. Then
the extension $L^h|K(x)^h$ is proper and immediate. So by (\ref{fundeq})
its defect is equal to its degree and thus non-trivial. If $L^h|K(x)^h$
were an Artin-Schreier extension, we could employ the same methods as
described above to find a normal form that allows us to find a better
$x$ (i.e., one for which the degree $[L^h:K(x)^h]$ is smaller). But in
general, even if $L^h|K(x)^h$ is separable, it will not necessarily be a
tower of Artin-Schreier extensions. Note that because its degree is a
prime, an Artin-Schreier extension does not admit any proper
subextensions; such an extension is called \bfind{minimal}. This leads
us to the following question: what is the structure of minimal
subextensions of such extensions $L^h|K(x)^h$?

Since $L^h|K(x)^h$ is immediate, but every finite subextension of
$K(x)^r|K(x)^h$ (where $K(x)^r:=(K(x)^h)^r\,$) is defectless by part e)
of Theorem~\ref{ram}, it follows that $L^h|K(x)^h$ is linearly disjoint
from $K(x)^r|K(x)^h$. Take any henselian field $(k,v)$. Then an
algebraic extension $k_1$ is called \bfind{purely wild} if $k_1|k$ is
linearly disjoint from $k^r|k$. Hence, our extension $L^h|K(x)^h$ is
purely wild. Our question is now answered by the following theorem,
which again shows the importance of $p$-polynomials and hence also of
additive polynomials. This theorem is due to Florian Pop (\fvklit{Pop}).

\begin{theorem}                             \label{mpwe}
Let $(k,v)$ be a henselian field of characteristic $p>0$ and $(k_1|k,v)$
a minimal purely wild extension. Then there exist an additive polynomial
${\cal A} (X)\in {\cal O}_k[X]$ and an element $\vartheta\in k_1$ such
that $k_1= k(\vartheta)$ and the $p$-polynomial ${\cal A} (X)-{\cal A}
(\vartheta)$ is the minimal polynomial of $\vartheta$ over $k$.
\end{theorem}
\n
It can be shown using Hensel's Lemma that if $k_1\ne k$, then ${\cal A}
(\vartheta)\notin {\cal O}_k\,$.

\pars
Using the additivity of the polynomial ${\cal A}$ like I used the
additivity of the Artin-Schreier polynomial $X^p-X$ before, it is indeed
possible to deduce a normal form that allows to find a better $x$.
Therefore, Theorem~\ref{mpwe} is an important ingredient in the proof of
Theorem~\ref{HR}.
Three sections of this paper are devoted to the previously unpublished
proof of Theorem~\ref{mpwe}. For the convenience of the reader,
$G$-modules and twisted homomorphisms are introduced in
Section~\ref{sectgmgc}. In Section~\ref{sectagfe}, a Galois theoretical
result of independent interest is proved. It is a generalization of the
theorem that I have already used above and that states that every
Galois extension of degree $p$ in characteristic $p$ is an
Artin-Schreier extension. Then I apply it in Section~\ref{sectmpwe} to
the situation of purely wild extensions and derive Theorem~\ref{mpwe}.

\pars
Theorem~\ref{mpwe} gains even more importance in conjunction with a
result of Matthias Pank (see [Ku--Pa--Ro]):
\begin{theorem}                             \label{pank}
Let $(K,v)$ be a henselian field. Then $K^r$ admits a field complement
$W$ in the algebraic closure $\tilde{K}$, that is, $W.K^r=\tilde{K}$ and
$W\cap K^r=K$. Every such complement $W$ is a maximal purely wild
extension of $K$. The quotient group $vW/vK$ is a $p$-group (where $p$
is the characteristic exponent of $Kv$), and the extension $Wv|Kv$ is
purely inseparable.
\end{theorem}
\n
Note that $(K,v)$ is a tame field if and only if $W=K$.

%
%
\subsection{Reason \#3: extremality and elementary properties of
additive polynomials}                                 \label{sectepap}
If $f$ is a polynomial in $n$
variables with coefficients in $K$, then we will say that $(K,v)$ is
\bfind{extremal with respect to $f$} if the set
\begin{equation}                            \label{defextr}
\{vf(a_1,\ldots,a_n)\mid a_1,\ldots,a_n\in K\}\>\subseteq vK\cup\{\infty\}
\end{equation}
has a maximum. This means that
\[\exists Y_1,\ldots,Y_n \forall X_1,\ldots,X_n vf(X_1,\ldots,X_n)\leq
vf(Y_1,\ldots,Y_n)\]
holds in $(K,v)$. It follows that being extremal with respect to $f$ is
an elementary property in the language of valued fields with parameters
from $K$. Note that the maximum is $\infty$ if and only if $f$ admits a
$K$-rational zero. A valued field $(K,v)$ is called \bfind{extremal} if
for all $n\in\N$, it is extremal with respect to every polynomial $f$ in
$n$ variables with coefficients in $K$. This notion is due to Ershov.
The property of being extremal can be expressed by a countable scheme of
elementary sentences (quantifying over the coefficients of all possible
polynomials of degree at most $n$ in at most $n$ variables). Hence, it
is elementary in the language of valued fields.

The following result was first stated by Delon in [Del], but the proof
contained gaps. The gaps were later filled by Ershov in [Er2]. I give
an alternative proof in [Ku8].

\begin{theorem}                             \label{extram}
A valued field is algebraically maximal if and only if it is extremal
with respect to every polynomial in one variable.
\end{theorem}

The following related results, also proved in [Ku8], illustrate
again the importance of additive and $p$-polynomials. First,
using Theorem~\ref{mpwe}, we can push the result stated in
Theorem~\ref{extram} even further:
\begin{theorem}                             \label{extram2}
A henselian valued field of characteristic $p>0$
is algebraically maximal if and only if it is
extremal with respect to every $p$-polynomial in one variable.
\end{theorem}

A polynomial ${\cal A}\in K[X_1,\ldots,X_n]$ in $n$ variables is called
\bfind{additive} if for all elements $a_1,\ldots,a_n,b_1,\ldots,b_n$ in
any extension field of $K$,
\[{\cal A}(a_1+b_1,\ldots,a_n+b_n)\;=\;{\cal A}(a_1,\ldots,a_n)\,+\,
{\cal A}(b_1,\ldots,b_n)\;.\]
In fact, if ${\cal A}$ is additive then
\[{\cal A}(X_1,\ldots,X_n)\;=\;\sum_{i=1}^{n} {\cal A}_i(X_i)\]
where
\[{\cal A}_i(X_i)\;:=\;{\cal A}(0,\ldots,0,X_i,0,\ldots,0)\]
are additive polynomials in one variable.
As before, a polynomial $f\in K[X_1,\ldots,X_n]$ in $n$ variables is
called \bfind{$p$-polynomial} if it is of the form ${\cal A}+c$ where
${\cal A}\in K[X_1,\ldots,X_n]$ is additive, and $c\in K$. From the
above we see that also every $p$-polynomial is a sum of $p$-polynomials
in one variable.

A valued field is called \bfind{inseparably defectless} if all purely
inseparable extensions have trivial defect. The following is proved in
[Ku8]:

\begin{theorem}                             \label{extrid}
A valued field $(K,v)$ of characteristic $p>0$
is inseparably defectless if and only if it is
extremal with respect to every $p$-polynomial of the form
\begin{equation}
b-\sum_{i=1}^{n} b_iX_i^p\>,\qquad n\in\N,\>b,b_1,\ldots,b_n\in K\;.
\end{equation}
\end{theorem}

Observe that again, all of these notions can be axiomatized by recursive
elementary axiom schemes.

\pars
I will now sketch the basic idea of the proof of Theorem~\ref{notc}.
Note that the image of a polynomial $f$ on a valued field $K$ has
the {\it optimal approximation property} in the sense of [Ku4] and
[Dr--Ku] if and only if $K$ is extremal with respect to $f-c$ for every
$c\in K$. Consequently,
\sn
{\it the images of all additive polynomials over $(K,v)$ have the
optimal approximation property if and only if $K$ is extremal with
respect to all $p$-polynomials over $K$.}
\sn
This holds in one variable as well as in several variables.

\pars
In [Ku4], I considered the following additive polynomial over
$\F_p((t))$:
\begin{equation}
X_0^p-X_0\,+\,tX^p_1\,+\,\ldots\,+t^{p-1}X^p_{p-1}\;.
\end{equation}
I showed that the image of this polynomial has the optimal approximation
property in $\F_p((t))$. Then I constructed an extension $(L,v)$ of
$\F_p((t))$ of transcendence degree $1$ which is henselian, defectless,
has value group a $\Z$-group, with $vt$ the smallest positive element,
and residue field $\F_p\,$, but the image of the above polynomial does
not have the optimal approximation property in $(L,v)$. This shows that
$\F_p((t))$ with its $t$-adic valuation is not an elementary
substructure of $(L,v)$. This yields Theorem~\ref{notc}.

Further, I proved in [Ku4] that in all maximal fields, the images of all
additive polynomials which satisfy a certain elementary condition have
the optimal approximation property. Maximal fields are interesting
objects in the model theory of valued fields because all maximal
immediate extensions of a valued field are maximal. So if we are
considering an elementary class of valued fields closed under maximal
immediate extensions (so far, this is the case for all classes of valued
fields without additional structure that play a role in model theory),
and if the Ax--Kochen--Ershov principle (\ref{AKEeq}) holds for this
class, then every field in the class should be elementarily equivalent
to all of its maximal immediate extensions. Therefore, the following
question is very important:

\begin{problem}
Is every maximal field of characteristic $p>0$
extremal with respect to every $p$-polynomial
in several variables? Is every maximal field extremal?
\end{problem}

Since every maximal field is algebraically maximal, Theorem~\ref{extram}
shows that it is at least extremal with respect to every polynomial in
one variable. To answer the first question to the affirmative, one would
have to eliminate the condition in the result mentioned above.

In [Ku4], I also construct an immediate function field $(F,v)$ of
transcendence degree $1$ over $(L,v)$ such that $(L,v)$ is not
existentially closed in $(F,v)$. Any maximal immediate extension
$(M,v)$ of $(F,v)$ is also a maximal immediate extension of $(L,v)$.
Since $(L,v)$ is not existentially closed in $(F,v)$, it is not
existentially closed in $(M,v)$. So it is not an elementary substructure
of $(M,v)$, and it cannot lie in an elementary class which has the
good properties discussed above.

The function field $F$ is generated over $L$ by two elements $x_0,x_1$
which satisfy an equation
\[x\;=\; x_0^p - x_0 + tx_1^p\]
where $x$ is an element in $L$ which is transcendental over $\F_p((t))$.
So the existential sentence
\[\exists X_0\exists X_1:\; x\;=\; X_0^p - X_0 + tX_1^p\]
holds in $F$. On the other hand, $L$ is constructed in such a way that
this sentence does not hold in $L$. This proves that $L$ is not
existentially closed in $F$ and, a fortiori, $(L,v)$ is not
existentially closed in $(F,v)$.

The function field $F|L$ shows the following interesting symmetry
between a generating Artin-Schreier extension and a generating purely
inseparable extension of degree $p$. On the one hand, we have the
Artin-Schreier extension
\[L(x_0,x_1)|L(x_1)\]
given by
\begin{equation}               \label{spc}
x_0^p - x_0 \>=\> x - tx_1^p\;.
\end{equation}
On the other hand we have the purely inseparable extension
\[L(x_0,x_1)|L(x_0)\]
given by
\[x_1^p \>=\> \frac{1}{t}(-x_0^p + x_0 + x)\;.\]
From equation (\ref{spc}) it is immediately clear that the function
field $L(x_0,x_1)$ becomes rational after a constant field extension by
$t^{1/p}$; namely
\[F(t^{1/p}) \>=\> L(t^{1/p})(x_0+t^{1/p}x_1)\;.\]
This shows that the base field $L$, not being existentially closed in
the function field $F$, becomes existentially closed in the function
field after a finite purely inseparable constant extension, although
this extension is linearly disjoint from $F|L$.

\pars
In our above example there also exists a separable constant field
extension $L'|L$ of degree $p$ such that $(F.L')^h$ is henselian
rational. To show this, we take a constant $d \in L$ and an element $a$
in the algebraic closure of $L$ satisfying
\[t \>=\> a^p -da\;,\]
and we put $L' = L(a)$. If we choose $d$ with a sufficiently high value,
then we will have that $vdax_1^p>0$. From this we deduce by Hensel's
Lemma that there is an element $b\in L'(x_1)^h$ such that $b^p - b =
-dax_1^p$. If we put $z = x_0 + ax_1 +b \in L'(x_0,x_1)^h$, we get that
\[z^p - z \>=\> x - tx_1^p + a^px_1^p -ax_1 - dax_1^p
\>=\> x - ax_1 + (a^p - da - t)x_1^p \>=\> x - ax_1\;,\]
which shows that
\[x_1 \>\in\> L'(z)\;.\]
This in turn yields that $b\in L'(z)^h$ and consequently,
\[x_0 \>=\> z -ax_1 -b\in L'(z)^h\;.\]
Altogether, we have proved that
\[L'(x_0,x_1)^h \>=\> L'(z)^h\]
is henselian rational.

\pars
Let us discuss one more problem about the model theory of $\F_p((t))$
that becomes visible through our above example. It can be shown that for
every $k\in\N$, the sentence
\[\forall X\exists X_0,\ldots,X_{p^k-1}\,,Y:\; X\>=\>X_0^{p^k}-X_0+
tX_1^{p^k}+\,\ldots\,+t^{p^k-1}X^{p^k}_{p^k-1}\,+\,Y\>\wedge\>vY\geq 0\]
holds in $\F_p((t))$ as well as in every maximal field which satisfies
axiom system (\ref{AS1}). On the other hand, given any $n\in\N$, the
construction of $(L,v)$ can be modified in such a way that for some $k$
the above sentence does not hold in $(L,v)$ and that the smallest
extension of $(L,v)$ within any maximal immediate extension in which
that sentence holds is at least of transcendence degree $n$ over $L$.
This is in drastic contrast to the tame behaviour shown by tame fields:
\sn
{\it If $(M|K,v)$ is an immediate extension, $(M,v)$ is a tame field and
$K$ is relatively algebraically closed in $M$, then also $(K,v)$ is a
tame field.}
\sn
(For a proof, see [Ku11].) This property of tame fields is used in an
essential way in the proof of Theorem~\ref{mttame} in order to reduce to
immediate extensions of transcendence degree 1 (so that Theorem~\ref{HR}
can be applied). Apparently, in the case of fields elementarily
equivalent to $\F_p((t))$, we have to succeed without this tool.

\parm
For the construction of the field $L$, I needed a handy criterion for
defectless fields. The following is proved in [Ku8]:
\begin{theorem}                             
A valued field of positive characteristic is henselian and defectless
if and only if it is separable-algebraically maximal and inseparably
defectless.
\end{theorem}
\n
The proof uses a classification of Artin-Schreier extensions with
non-trivial defect according to whether they can be obtained as a
deformation of a purely inseparable extension with non-trivial defect,
or not (cf.\ [Ku8]). This classification is also of independent
interest. For instance, S.~D.~Cutkosky and O.~Piltant [Cu--Pi] give an
example of a tower of two Artin-Schreier extensions with non-trivial
defect of a rational function field in which a certain form of
``relative resolution'' fails. It would be interesting to know whether
such properties depend on the classification. In [Ku6], valued
rational function fields are constructed which allow an infinite tower
of Artin-Schreier extensions with non-trivial defect, but it is not
clear whether one can obtain both sorts of extensions. The
classification may also be important for the characterization of all
valued fields whose maximal immediate extensions are finite (cf.\ [V]
for the background).

%
%
\subsection{But what about extremality for all polynomials?}
Let us come back to the question whether every maximal field is
extremal. We know the answer in the case of discrete valued fields:

\begin{theorem}
If $(K,v)$ is a henselian defectless field with value group isomorphic
to $\Z$, then $(K,v)$ is extremal.
\end{theorem}

In [Del], Delon deduced this from the work of Greenberg [Gre]. An
elegant model theoretic proof was given by Ershov. The theorem implies
that in particular, $(\F_p((t)),v_t)$ is extremal. It also implies that
every henselian defectless field with value group isomorphic
to $\Z$ is extremal with respect to all $p$-polynomials in several
variables. An alternative proof for this fact can be found in [Dr--K].
It uses the local compactness of $\F_p((t))$. If this could be
eliminated in the case of maximal fields, we could at least prove that
every maximal field is extremal with respect to all $p$-polynomials in
several variables. This generates the following question:

\begin{problem}
If a henselian field of characteristic $p>0$ is extremal with respect to
all $p$-polynomials in several variables, does this imply that it is
extremal?
\end{problem}

Are $p$-polynomials representative for all polynomials when extremality
is concerned? Theorem~\ref{extram2} indicates that modulo henselization
this is true for polynomials in one variable. But can we associate
directly to every polynomial in one variable a $p$-polynomial in one
variable from which we can read off information about extremality? A
result of Kaplansky ([Ka1], Lemma~10), originally proved to be used in
the construction of one of the counterexamples to embeddability in power
series fields, shows that this can be done over every henselian field
with archimedean value group. Using technical machinery developed in
[Ku1], this result can be generalized (the proof is implicit in [Ku1]):

\begin{proposition}                         \label{addprepr}
Let $(K,v)$ be a henselian field and $(a_\rho)_{\rho<\lambda}$ a pseudo
Cauchy sequence in $K$ without a limit in $K$. Pick a polynomial $f$ of
minimal degree such that the value $vf(a_\rho)$ is not ultimately fixed.
Then there is an additive polynomial ${\cal A}\in K[X]$ such that for
all large enough $\rho$,
\[v(f(a_\rho)\,-\,{\cal A}(a_\rho))\> >\>vf(a_\rho)\]
(which in particular implies that $vf(a_\rho)=v{\cal A}(a_\rho)$).
\end{proposition}

\begin{problem}
Is Proposition~\ref{addprepr} also true for polynomials in
several variables?
\end{problem}
\n
If this were not the case, then it would destroy our hope to capture the
complete theory of $\F_p((t))$ by adjoining axioms about extremality
with respect to additive polynomials to axiom system (\ref{AS1}). That
would mean that additive polynomials are important but do not tell us
all the missing information about $\F_p((t))$.

It should be mentioned that the case of several variables is very much
different from the case of one variable, and there is not much hope of
treating it by induction on the number $n$ of variables starting with
$n=1$. Indeed, if $(L|K,v)$ is an immediate extension generated by a
polynomial $f$, then a pseudo Cauchy sequence of algebraic type can be
constructed with respect to which the value of $f$ is not fixed (for
these notions, see [Ka1]). This has been done in [Er2] and in [Ku8]. A
similar procedure is {\it not} known for the case of several variables.

%
%
\subsection{Concluding remarks about valued $K[\varphi]$-modules}
Van den Dries' question can be reformulated as:
{\it Determine the model theory of valued $K[\varphi]$-modules.}
What do we mean by a ``valued module''? There are some notions of
``valued module'' in the literature, but as far as I know they do not
cover the case we are interested in. Basically, one could define a
``valued module'' to be a module which also has the structure of a
valued abelian group. But without any further assumptions on the
compatibility between module structure and valuation, this would not
lead us far. So we have to choose axioms for the compatibility that
cover the case we are interested in. I have done this in [Ku2], but
these axioms are not yet in a very satisfactory form. Although the
structure of $K[\varphi]$-modules can be nasty when $K$ is not perfect,
there is still the valuation on them and it appears that with an
appropriate choice of axioms one can tame these modules. Indeed, a first
answer to van den Dries' question was given by his student Thomas Rohwer
who proved in his thesis [Roh] the following results:

\begin{theorem}                             \label{rohw}
The elementary theory of $\F_p((t))$ as an $\F_p((t))[\varphi]$-module
with a predicate for $\F_p[[t]]$ is model complete.
The elementary theory of $\F_p((t))$ as an $\F_p(t)[\varphi]$-module
with a predicate for $\F_p[[t]]$ is decidable.
\end{theorem}
\n
It should be noted that Pheidas and Zahidi [Ph--Za] prove analogous
results for $\F_p[t]$ as an $\F_p[t][\varphi]$-module.

\pars
Theorem~\ref{rohw} immediately leads to a number of questions:

\begin{problem}
What do Rohwer's results tell us about the model theory of the valued
field $\F_p((t))$?
\end{problem}

\begin{problem}
Does the elementary theory of $\F_p((t))$ as an $\F_p((t)) [\varphi]$-
or $\F_p(t)[\varphi]$-module admit quantifier elimination in some
natural language?
\end{problem}

Rohwer works with predicates $V_i$ that are interpreted by the sets of
all elements of value $\geq i$. This gives less information than a
binary predicate $P(x,y)$ interpreted by $vx\leq vy$ (``valuation
divisibility'').

\begin{problem}
What are the model theoretic properties of the elementary
theory of $\F_p((t))$ as a valued $\F_p((t))[\varphi]$- or
$\F_p(t)[\varphi]$-module in a language which includes a binary
predicate for valuation divisibility?
\end{problem}

\begin{problem}
What is the structure of extensions of valued $K[\varphi]$-modules? Can
one prove Ax--Kochen--Ershov principles for valued $K[\varphi]$-modules?
\end{problem}

An important tool in the model theory of valued fields is Kaplansky's
well known result that a valued field is maximal if and only if every
pseudo Cauchy sequence in this field has a limit (cf.\ [Ka1]). One can
ask the same question for other valued structures. In the case of valued
modules with value-preserving scalar multiplication, the corresponding
result is already in the literature. For the case of valued modules with
the above mentioned axioms that cover the case of the valued
$K[\varphi]$-module $\F_p((t))$, I proved the corresponding result in
[Ku2]. Together with Rohwer's work, this seems to be a good start
towards a comprehensive study of valued $K[\varphi]$-modules, including
a full answer to van den Dries' question, but quite a bit of work
remains to be done.

%
%
%
\section{Characterization of additive polynomials}  \label{sectchar}
In this section we give the basic characterizations of
additive polynomials and prove Theorem~\ref{thmaddpol}.

\begin{lemma}                               \label{lemgXY}
Take $f\in K[X]$ and consider the following polynomial in two variables:
\begin{equation}                            \label{gXY}
g(X,Y)\;:=\;f(X+Y)-f(X)-f(Y)\;.
\end{equation}
If there is a subset $A$ of cardinality at least $\deg f$ in some
extension field of $K$ such that $g$ vanishes on $A\times A$, then $f$
is additive and of the form (\ref{formaddpol}).
\end{lemma}
\begin{proof}
Assume that there is a subset $A$ of cardinality at least $\deg f$ in
some extension field of $K$ such that $g$ vanishes on $A\times A$. Take
$L$ to be any extension field of $K$. By field amalgamation, we may
assume that $A$ is contained in an extension field $L'$ of $L$. For all
$c\in L'$, the polynomials $g(c,Y)$ and $g(X,c)$ are of lower degree
than $f$. This follows from their Taylor expansion. Assume that there
exists $c\in L$ such that $g(c,Y)$ is not identically 0. Since $A$ has
more than $\deg g(c,Y)$ many elements, it follows that there must be
$a\in A$ such that $g(c,a)\not=0$. Consequently, $g(X,a)$ is not
identically 0. But since $A$ has more than $\deg g(X,a)$ many elements,
this contradicts the fact that $g(X,a)$ vanishes on $A$. This
contradiction shows that $g(c,Y)$ is identically 0 for all $c\in L$.
That is, $g$ vanishes on $L\times L$. Since this holds for all extension
fields $L$ of $K$, we have proved that $f$ is additive.

By what we have shown, $g(c,Y)$ vanishes identically for every $c$ in
any extension field of $K$. That means that the polynomial $g(X,Y)\in
K(Y)[X]$ has infinitely many zeros. Hence, it must be identically 0.
Write $f=d_nX^n+\ldots+d_0\,$. Then $g(X,Y)$ is the sum of the forms
$d_j(X+Y)^j-d_jX^j-d_jY^j$ of degree $j$, $1\leq j\leq \deg f$. Since
$g$ is identically 0, the same must be true for each of these forms and
thus for all $(X+Y)^j-X^j-Y^j$ for which $d_j\not=0$. But $(X+Y)^j-X^j
-Y^j\equiv 0$ can only hold if $j$ is a power of the characteristic
exponent of $K$. Hence, $d_j=0$ if $j$ is not a power of $p$. Setting
$c_i:=d_{p^i}$, we see that $f$ is of the form (\ref{formaddpol}).
\end{proof}

\n
{\bf Proof of Theorem~\ref{thmaddpol}:} \
Suppose that $f\in K[X]$ is additive. Then the polynomial $g$ defined in
(\ref{gXY}) vanishes on every extension field $L$ of $K$. Choosing $L$
to be infinite and taking $A=L$, we obtain from the foregoing
lemma that $f$ is of the form (\ref{formaddpol}).

Conversely, for every $i\in\N$, the mapping $x\mapsto x^{p^i}$ is a
homomorphism on every field of characteristic exponent $p$. Hence, every
polynomial $c_iX^{p^i}$ is additive, and so is the polynomial
$\sum_{i=0}^{m}c_i X^{p^i}$.                                   \QED

\begin{corollary}                               \label{charaddpol}
Take $f\in K[X]$.\n
a) \ If $f$ is additive, then the set of its roots in the algebraic
closure $\tilde{K}$ of $K$ is a subgroup of the additive group of
$\tilde{K}$. Conversely, if the latter holds and $f$ has no multiple
roots, then $f$ is additive.
\n
b) \ If $f$ satisfies condition~(\ref{fa+b}) on a field with at least
$\deg f$ many elements, then $f$ is additive.
\end{corollary}
\begin{proof}
a): \ If $f$ is additive and $a,b$ are roots of $f$, then $f(a+b)=f(a)
+f(b)=0$; hence $a+b$ is also a root. Further, $f(0)=f(0+0)=f(0)+f(0)$
shows that $0=f(0)=f(a-a)=f(a)+f(-a)=f(-a)$, so $0$ and $-a$ are also
roots. This shows that the set of roots of $f$ form a subgroup of
$(\tilde{K},+)$.

Now assume that the set $A$ of roots of $f$ forms a subgroup of
$(\tilde{K},+)$, and that $f$ has no multiple roots. The latter implies
that $A$ has exactly $\deg f$ many elements. Since $A+A=A$, the
polynomial $g(X,Y)=f(X+Y)-f(X)-f(Y)$ vanishes on $A\times A$. Hence by
Lemma~\ref{lemgXY}, $f$ is additive.
\mn
b): \ This is an immediate application of Lemma~\ref{lemgXY}.
\end{proof}

\pars
\begin{exercise}\n
a)\ \ Let $K$ be any finite field. Give an example of a polynomial
$f\in K[X]$ which is not additive but induces an additive mapping on
$(K,+)$.\n
b)\ \ Show that the second assertion in part a) of
Corollary~\ref{charaddpol} fails if we drop the condition that $f$ has
no multiple roots. Replace this condition by a suitable condition on the
multiplicity of the roots.\n
c)\ \ Deduce Corollary~\ref{charaddpol} from the theorem of \ind{Artin}
as cited in \fvklit{L}, VIII, \S11, Theorem~18.
\end{exercise}

%
%
%
\section{Rings of additive polynomials}              \label{sectring}
This section is devoted to the structure of rings of additive
polynomials. Euclidean division is discussed in the following

{\bf Proof of Theorem~\ref{Kphi}:} \
Take $s=\sum_{i=0}^{m}c_i \varphi^i$ and $s'=\sum_{i=0}^{n}d_i
\varphi^i$. If $\deg s'<\deg s$, then we set $q=0$ and $r=s'$. Now
assume that $\deg s'=n\geq m=\deg s$. Then
\[\deg(s'- d_n c_m^{-p^{n-m}}\varphi^{n-m}\, s)\leq n-1<\deg s'\;.\]
Now take $q\in K[\varphi]$ such that $\deg (s'-qs)$ is minimal. Then
$\deg (s'-qs)<\deg s$. Otherwise, we could apply the above to $s'-qs$ in
the place of $s'$, finding some $q'\in R$ such that $\deg (s'- (q+q')s)
= \deg (s'-qs-q's)<\deg (s'-qs)$ contradicting the minimality of $q$.
Setting $r=s'-qs$, we obtain $s'=qs+r$ with $\deg r <\deg s$. We have
proved that $K[\varphi]$ is left euclidean. If $K$ is perfect, hence
$K=K^{p^m}$, then we also have
\[\deg(s'- s\, (c_m^{-1}d_n)^{1/p^m} \varphi^{n-m})\leq n-1<\deg s'\;,\]
and in the same way as above one deduces that $K[\varphi]$ is right
euclidean.

Now assume that $K$ is not perfect and choose some element $c\in K$
not admitting a $p$-th root in $K$. Then $K^p\cap cK^p=\{0\}$ and
\[\varphi K[\varphi]\,\cap\, c\varphi K[\varphi] = \{0\}\]
since every nonzero additive polynomial in the set $\varphi K[\varphi]$
has coefficients in $K^p$ whereas every nonzero additive polynomial in
$c\varphi K[\varphi]$ has coefficients in $cK^p$.
                                                          \QED

\pars
\begin{remark}                              \label{remCohn}
Let us state some further properties of the ring $K[\varphi]$ which
follow from Theorem~\ref{Kphi}. More generally, let $R$ be any left
principal ideal domain. Then $R$ is a left free ideal ring (fir), and it
is thus a semifir, i.e., every finitely generated left or right ideal is
free of unique rank (note that this property is left-right symmetrical,
cf.~\fvklit{C2}, Chapter 1, Theorem~1.1). Consequently, every finitely
generated submodule of a (left or right) free $R$-module is again free,
cf.~\fvklit{C2}, Chapter 1, Theorem~1.1$\,$. On the other hand, every
finitely generated torsion free (left or right) $R$-module is embeddable
in a (finitely generated) free $R$-module if and only if $R$ is right
Ore, cf.~\fvklit{C2}, Chapter 0, Corollary~9.5 and \fvklit{Ge},
Proposition 4.1$\,$. Being a semifir, $R$ is right Ore if and only if it
is a right Bezout ring. But if $R$ is not right Ore, then it contains
free right ideals of arbitrary finite or countable rank, and $R$ is thus
not right noetherian, cf.~\fvklit{C2}, Chapter 0, Proposition~8.9 and
Corollary~8.10$\,$. Every projective (left or right) $R$-module is free,
cf.~\fvklit{C2}, Chapter 1, Theorem~4.1$\,$. A right $R$-module is flat
if and only if it is torsion free, and a left $R$-module $M$ is flat if
and only if every finitely generated submodule of $M$ is free,
cf.~\fvklit{C2}, Chapter 1, Corollary~4.7 and Proposition~4.5$\,$. In
view of the above, the latter is the case if and only if every finitely
generated submodule of $M$ is embeddable in a free $R$-module. Further,
a left $R$-module $M$ is flat if and only if for every $n\in \N$ and all
right linearly independent elements $r_1,\ldots,r_n\in R$,
\[\forall x_1,\ldots,x_n\in M:\;\; \sum r_i x_i = 0 \,\Rightarrow\,
\forall i:\> x_i = 0\;,\]
cf.~\fvklit{C1}, Chapter 1, Lemma~4.3$\,$. As a semifir, $R$ is a
coherent ring. Finally, since $R$ is left Ore, it can be embedded in a
skew field of left fractions, cf.~\fvklit{C2}, Chapter 0,
Corollary~8.7$\,$.
\end{remark}

Note that in particular, the above shows that all finitely generated
torsion free (left or right) $K[\varphi]$-modules are free if and only
if $K$ is perfect, that is, $K[\varphi]$ is euclidean on both sides.

\pars
\begin{exercise}
Describe the relation of the degree functions on $K[X]$ and $K[\varphi]$
via the correspondence (\ref{corraddpol}), giving thereby a proof of
$\deg rs = \deg r +\deg s$. Show that it also satisfies $\deg r+s \leq
\max\{\deg r,\deg s\}$ with equality holding if $\deg r\not=\deg s$.
Can it be transformed into a valuation?
\end{exercise}

%
%
\section{Frobenius-closed bases of function fields}  \label{sectFrcb}
In this section, we prove the existence of Frobenius-closed bases of
algebraic function fields $F|K$ of transcendence degree 1, and
exhibit the connection between their existence and the structure of $F$
as a $K[\varphi]$-module (for arbitrary transcendence degree).

Take an arbitrary extension $F|K$ of fields of characteristic $p > 0$.
Recall that a $K$-basis $B$ of $F$ is called \bfind{Frobenius-closed} if
$B^p \subset B$, where $B^p = \{b^p\mid b\in B\}=\varphi B$. In [Ku9] we
need Frobenius-closed bases because they have the following property:
\begin{lemma}                  \label{Frpp}
Take a Frobenius-closed basis $z_j$, $j\in J$, of $F|K$. If the sum
\[s = \sum_{i\in I} c_i z_i\,,\;\; c_i\in K\,,
\;I\subset J\mbox{\ finite}\]
is a $p$-th power, then for every $i\in I$ with $c_i\not= 0$, the basis
element $z_i$ is a $p$-th power of a basis element.
\end{lemma}
\begin{proof}
Assume that
\[s = \left(\sum_{j\in J_0} {c_j}' z_j\right)^p ,\;\; {c_j}'\in K\]
where $J_0\subset J$ is a finite index set. Then
\[\sum_{i\in I} c_i z_i = s = \sum_{j\in J_0} ({c_j}')^p z_j^p\]
where the elements $z_j^p$ are also basis elements by hypothesis,
which shows that every $z_i$ which appears on the left hand side
(i.e., $c_i \not= 0$) equals a $p$-th power\ $z_j^p$ appearing on the
right hand side.
\end{proof}

\pars
We will show the existence of Frobenius-closed bases for algebraic
function fields of transcendence degree $1$ over a perfect field of
characteristic $p>0$, provided that $K$ is relatively algebraically
closed in $F$. We first prove the following:
\begin{lemma}
If $F$ is an algebraic function field of transcendence degree $1$ over
an algebraically closed field $K$ of arbitrary characteristic and $q$ is
an arbitrary natural number $>1$, then there exists a basis of $F|K$
which is closed under $q$-th powers.
\end{lemma}
\n
If $F= K(x)$ is a rational function field, then our lemma follows from
the \bfind{Partial Fraction Decomposition}: Every element $f\in
F$ has a unique representation
\[f= c + \sum_{n>0}c_nx^n + \sum_{a\in K}\sum_{n>0}
c_{a,n}\frac{1}{(x-a)^n}\]
where only finitely many of the coefficients $c,c_n,c_{a,n}\in K$
are nonzero. If we put
\[t_a = \frac{1}{x-a}\;,\;\; t_{\infty} = x\]
then it follows that the elements
\[1,\;t_a^n \mbox{\ \ with\ \ } a\in K\cup\{\infty\},\; n\in \N\]
form a $K$-basis of $F$; this basis has the property that {\bf every}
power of a basis element is again a basis element.

For general function fields the Partial Fraction Decomposition remains
true in a modified form (according to Helmut Hasse) that we shall
describe now. At this point, we need the Riemann-Roch Theorem. In order
to apply it, we have to introduce some notation. In what follows, we
always assume that $K$ is relatively algebraically closed in $F$. A
\bfind{divisor} of $F|K$ is an element of the (multiplicatively written)
free abelian group generated by all places of $F|K$. (By a place of
$F|K$ we mean a place of $F$ which is trivial on $K$, i.e., $P|_K=
\mbox{id}$. We identify equivalent places.) The places themselves are
called \bfind{prime divisors}. A divisor may thus be written in the form
\[
A=\prod_{P} P^{\,v_P A}
\]
where the product is taken over all places of $F|K$ and the $v_P A$ are
integers, only finitely many of them nonzero. The {\bf degree of a
non-trivial place}\index{degree of a place} $P$ of $F|K$, denoted by
\gloss{$\deg P$}, is defined to be the degree $[FP:K]\,$ (which is
finite since $F|K$ is an algebraic function field in one variable).
Accordingly, the \bfind{degree of a divisor} $A$, denoted by
\gloss{$\deg A$}, is defined to be the integer $\sum_{P} v_P A \cdot
\deg P\,$. By the symbol ``$v_P$'' we will also denote the valuation
on $F$ which is associated with the place $P$. Following the notation of
\fvklit{F--Jr}, we set
\[{\cal L}(A)\>:=\>\{f\in F\mid v_P f\geq -v_P A \mbox{ for all places
$P$ of } F|K\}\]
is a $K$-vector space. Indeed, $0\in {\cal L}(A)$ since $v_P 0=\infty>
-v_P A$ for all places $P$ of $F|K$. Further, $v_P (K^\times) =\{0\}$,
hence $\forall c\in K^\times:\>v_P (cf)=v_P f$ for all $P$, so $f\in
{\cal L} (A)$ implies $cf\in {\cal L}(A)$. Finally, if $f,g\in {\cal L}
(A)$, then $v_P(f-g)\geq \min\{v_Pf, v_Pg\}\geq -v_P A$ for all $P$,
hence $f-g\in {\cal L}(A)$. We write
\[\dim A\>:=\>\dim_K {\cal L}(A)\;.\]
The divisor $A$ determines bounds for the zero and pole orders of the
algebraic functions in ${\cal L}(A)$. For example, if $A=P^n$ with $n$ a
natural number, then $f\in {\cal L}(A)$ if and only if $f$ has no pole
at all (in which case it is a constant function) or has a pole at $P$ of
pole order at most $n= v_P A$.

\begin{theorem} {\bf (Riemann-Roch)}\n
Let $F|K$ be an algebraic function field in one variable with $K$
relatively algebraically closed in $F$. There exists a smallest
non-negative integer $g$, called the genus of $F|K$, such that
\[\dim A \geq \deg A -g +1\]
for all divisors $A$ of $F|K$. Furthermore,
\[\dim A = \deg A -g +1\]
whenever $\deg A >2g-2$.
\end{theorem}
\n
For a proof, see \fvklit{Deu}.
\pars
Let $P_{\infty}$ be a fixed place of $F|K$ and $R^{\infty}$ the ring of
all $f\in F$ which satisfy $v_P f\geq 0$ for every $P\not= P_{\infty}$.
The following is an application of the Riemann-Roch Theorem:
\begin{corollary}
For every $P\not=P_{\infty}$ there exists an element $t_P\in F$ such that
\begin{eqnarray*}
v_P t_P & = & -1\\
v_Q t_P & \geq & 0\;\;\mbox{\ \ for\ \ } Q\not= P,P_{\infty}\;.
\end{eqnarray*}
\end{corollary}
\begin{proof}
If we choose $n\in \N$ as large as to satisfy $n\deg P_{\infty}>2g-2$,
then by the Riemann-Roch Theorem,
\[\dim (PP_{\infty}^n) = \deg P + n\deg P_{\infty} - g + 1
>n\deg P_{\infty} - g + 1 = \dim P_{\infty}^n\;.\]
Hence there is an element $t_P\in {\cal L}(PP_{\infty}^n)\setminus
{\cal L}(P_{\infty}^n)$. This element has the required properties.
\end{proof}

We return to the proof of our lemma, assuming that $K$ is algebraically
closed. Hence, $K$ is the residue field of every place $P$ of $F|K$
(that is, $\deg P=1$). Every $t_P$ of the foregoing corollary is the
inverse of a uniformizing parameter for $P$. Every $f\in F$ can be
expanded $P$-adically with respect to such a uniformizing parameter, and
the principal part appearing in this expansion has the form
\[h_P(f) = \sum_{n>0} c_{P,n} t_P^n\;,\]
where only finitely many of the coefficients $c_{P,n}\in K$ are nonzero,
namely $n\leq - v_P f$.\n
By construction, $t_P$ has only a single pole $\not= P_{\infty}$ and this
pole is $P\,$; the same holds for $h_P(f)$ (if $h_P(f)\not= 0$).
Consequently,
\[h = f - \sum_{P\not= P_{\infty}} h_P(f)\]
has no pole other than $P_{\infty}$ and is thus an element of $R^{\infty}$.
We have shown that $f$ has a unique representation
\[f= h+ \sum_{P\not=P_{\infty}} \sum_{n>0} c_{P,n}^{ } t_P^n\]
with coefficients $c_{P,n}\in K$ and an element $h\in R^{\infty}$. This
shows that the elements
\[t^n_P \mbox{\ \ with\ \ } P\not= P_{\infty},\; n\in \N\]
form a $K$-basis of $F$ modulo $R^{\infty}$ which has the property that
every power of a basis element is again a basis element.
\par\medskip
Now it remains to show that $R^{\infty}$ admits a basis which is closed
under $q$-th powers.
An integer $n\in\N$ is called \bfind{pole number}\ of $P_{\infty}$ if
there exists $t_n\in R^{\infty}$ such that $v_{P_{\infty}} t_n = -n$.
Let $H_{\infty} \subseteq \N$ be the set of all pole numbers. Fixing a
$t_n$ for every $n\in H_{\infty}\,$, we get a $K$-basis
\[1,\; t_n \mbox{\ \ with\ \ } n\in H_{\infty}\]
of $R^{\infty}$. To get a basis which is closed under $q$-th powers, we
have to carry out our choice as follows:

Observe that $H_{\infty}$ is closed under addition; in particular
\[qH_{\infty} \subset H_{\infty}\;.\]
For every $m\in H_{\infty}\setminus qH_{\infty}$ we choose an arbitrary
element $t_m\in R^{\infty}$ with $v_{P_{\infty}} t_m = -m$. Every
$n \in H_{\infty}$ can uniquely be written as
\[n = q^{\nu} m \mbox{\ \ where\ \ } \nu\geq 0\mbox{\ \ and\ \ } m \in
H_{\infty}\setminus qH_{\infty}\;.\]
Accordingly we put
\[t_n = t_m^{\,q^{\nu}}\]
which implies
\[v_{P_{\infty}} t_n = q^{\nu}\cdot v_{P_{\infty}} t_m = -q^{\nu} m
= -n\;.\]
This construction produces a $K$-basis
\[1,\; t_m^{\,q^{\nu}} \mbox{\ \ with\ \ } m\in H_{\infty}\setminus
qH_{\infty}, \; \nu\geq 0\]
of $R^{\infty}$, which is closed under $q$-th powers. This concludes the
proof of our lemma.

\parm
For the generalization of this lemma to perfect ground fields of
characteristic $p>0$ we have to choose $q = p$.
\sn
{\bf Proof of Theorem~\ref{Fc}:} \ We have to prove:
\sn
{\it Let $F$ be an algebraic function field of transcendence degree $1$
over a perfect field $K$ of characteristic $p>0$. If $K$ is relatively
algebraically closed in $F$, then there exists a Frobenius-closed basis
for $F|K$.}
\sn
If $K$ is not algebraically closed, we have to modify the proof of the
previous lemma since not every place $P$ of $K$ has degree 1. (Such a
modification is also necessary for the Partial Fraction Decomposition in
$K(x)$ if $K$ is not algebraically closed.) The modification reads as
follows:

For every place $P$ of $F|K$, let
\[d_P = \deg P = [FP:K]\]
be the degree of $P$. For every $P\not= P_{\infty}$ we choose
elements $u_{P,i}\in R^{\infty}$, $1\leq i\leq d_P$, such that their
residues $u_{P,1}P,...,u_{P,d_P}P$ form a $K$-basis of $FP$. We note
that for every $\nu\geq 0$, the $p^{\nu}$-th
powers $u_{P,i}^{p^{\nu}}$ of these elements have the same
property: their $P$-residues also form a $K$-basis of $FP$ since $K$
is perfect.\n
We write every $n\in \N$ in the form
\[n = p^{\nu} m \mbox{\ \ with\ \ } m\in \N ,\; (p,m) = 1,\;\nu\geq 0\]
and observe that the elements
\[u_{P,i}^{p^{\nu}}\,t_P^n \mbox{\ \ with\ \ } P\not= P_{\infty},\;
n\in \N,\; 1\leq i\leq d_P\]
form a Frobenius-closed $K$-basis of $F$ modulo $R^{\infty}$.
\par\medskip
It remains to construct a Frobenius-closed $K$-basis of $R^{\infty}$.
This is done as follows: We consider the $K$-vector spaces
\[{\cal L}_n = {\cal L}(P_{\infty}^n)=\{x\in F\mid v_{P_{\infty}} x\geq
-n\mbox{ and } v_P(x)\geq 0 \mbox{ for } P\not= P_{\infty}\}\;.\]
By our assumption that $K$ is relatively algebraically closed in $F$, we
have ${\cal L}_0 = K$. Further,
\[R^{\infty}\>=\>\bigcup_{n\in\N} {\cal L}_n\;.\] We set
\[d_{\infty,n} \>:=\> \dim {\cal L}_n/{\cal L}_{n-1}\> \geq \> 0\;.\]
(Note that by the Riemann-Roch Theorem, $d_{\infty,n}=[FP_{\infty}:K]$
for large enough $n$; cf.~the proof of the above corollary.) Now for $n
= 1,2,...$ we shall choose successively basis elements $t_{n,i} \in
{\cal L}_n$ modulo ${\cal L}_{n-1}\,$. Then the elements
\[1,\; t_{n,i} \mbox{\ \ with\ \ } n\in\N,\; 1\leq i\leq d_{\infty,n}\]
form a $K$-basis of $R^{\infty}$. To obtain that this basis is
Frobenius-closed, we organize our choice as follows:

If $n = pm$, the $p$-th powers $t^p_{m,i}\in {\cal L}_n$ are linearly
independent modulo ${\cal L}_{p(m-1)}$ and even modulo ${\cal L}_{pm-1}
= {\cal L}_{n-1}$. This fact follows from our hypothesis that $K$ is
perfect: the existence of nonzero elements $c_i\in K$ with $\sum c_i
t^p_{m,i} \in {\cal L}_{pm-1}$, i.e., $v_{P_{\infty}}\sum c_i t^p_{m,i}
>-pm$, would yield $v_{P_{\infty}} \sum c_i^{1/p} t_{m,i}>-m$, hence
$\geq -m+1$, showing that $\sum c_i^{1/p} t_{m,i} \in {\cal L}_{m-1}$,
which is a contradiction. In our choice of the elements $t_{n,i}$ we are
thus free to take all the elements $t_{m,i}^p$ and to extend this set to
a basis of ${\cal L}_n$ modulo ${\cal L}_{n-1}$ by arbitrary further
elements, if necessary (for $n$ large enough, the elements $t_{m,i}^p$
will already form such a basis). This procedure guarantees that the
$p$-th power of every basis element $t_{m,i}$ is again a basis element,
namely equal to $t_{pm,j}$ for suitable $j$. Hence a basis constructed
in this way will be Frobenius-closed.                         \QED

\parm
Let $F|K$ be an arbitrary extension of fields of characteristic $p>0$.
Both $F$ and $K$ are $K[\varphi]$-modules, and so is the quotient module
$F/K\,$. Suppose that $F/K$ is a free $K[\varphi]$-module. Then it
admits a $K[\varphi]$-basis. Let $B_0\subset F$ be a set of
representatives for such a $K[\varphi]$-basis of $F/K$. It follows that
\[B \,=\, \bigcup_{n=0}^{\infty} B_0^{p^n}\,\cup\,\{1\} \,=\,
\bigcup_{n=0}^{\infty} \varphi^n B_0 \,\cup\,\{1\}\]
is a set of generators of the $K$-vector space $F$. By our construction
of $B$, every $K$-linear combination of elements of $B\setminus \{1\}$
may be viewed as a $K[\varphi]$-linear combination of elements of
$B_0\,$. This shows that the elements of $B$ are $K$-linearly
independent, and $B$ is thus a Frobenius-closed basis of $F|K$. Note
that $B_0$ is the basis of a free $K[\varphi]$-submodule $M$ of $F$
which satisfies $F=M\oplus K\,$.

The converse to this procedure would mean to extract a
$K[\varphi]$-basis $B_0$ from a Frobenius-closed $K$-basis $B$. But
$B_0$ can only be found if for every element $b\in B\setminus \{1\}$
there is some element $b_0$ which is not a $p$-th power in $F$ and such
that $b=b_0 ^{p^n}$ for some $n\in\N\cup\{0\}\,$. This will hold if no
element of $F\setminus K$ has a $p^n$-th root for every $n\in\N$.
\begin{lemma}
If $F|K$ is an algebraic function field (of arbitrary transcendence
degree), and if $K$ is relatively algebraically closed in $F$, then no
element of $F\setminus K$ has a $p^n$-th root for every $n\in\N$.
\end{lemma}
\begin{proof}
Let $f\in F\setminus K$. Since $K$ is relatively algebraically closed in
$F$, we know that $f$ is transcendental over $K$. So we may choose a
transcendence basis ${\cal T}$ of $F|K$ containing $f$. We may choose a
$K$-rational valuation $v$ on the rational function field $K({\cal T})$
such that the values of all elements in ${\cal T}$ are rationally
independent (cf., e.g., [Ku7]). This yields that
$vK({\cal T})= \bigoplus_{t\in {\cal T}}\Z vt$. In particular, $vf$ is
not divisible by $p$ in $vK({\cal T})$. Since $F|K({\cal T})$ is finite,
the same is true for $(vF:vK({\cal T}))$ by the fundamental inequality
(\ref{fundineq}). This yields that there is some $n\in\N$ such that $vf$
is not divisible by $p^n$ in $vF$. Hence, $f$ does not admit a $p^n$-th
root in $F$.
\end{proof}

This lemma shows that if $F|K$ is an algebraic function field with $K$
relatively algebraically closed in $F$, admitting a Frobenius-closed
basis $B$ and if we let $B_0$ be the set of all elements in $B$ which do
not admit a $p$-th root in $F$, then we obtain $B=\bigcup_{n=0}^{\infty}
B_0^{p^n}\,\cup\,\{1\}$. Since the elements of $B$ are $K$-linearly
independent, the elements of $B_0$ are $K[\varphi]$-linearly
independent over $K$. Moreover, $B_0$ is a set of generators of the
$K[\varphi]$-module $F$ over $K$. Hence, the set $B_0/K$ is a
$K[\varphi]$-basis of $F/K\,$. We have thus proved:
\begin{proposition}                               \label{F/Kfree}
Let $F|K$ be an algebraic function field (of arbitrary transcendence
degree), and $K$ relatively algebraically closed in $F$. Then $F$ admits
a Frobenius-closed $K$-basis if and only if $F/K$ is a free
$K[\varphi]$-module.
\end{proposition}
\n
This lemma shows that Theorem~\ref{Fc} implies Theorem~\ref{freeKfmod}.

\pars
\begin{problem}
Do Theorems~\ref{Fc} and~\ref{freeKfmod} also hold for transcendence
degree $>1$?
\end{problem}
\begin{problem}
Do Theorems~\ref{Fc} and~\ref{freeKfmod} also hold if the assumption
that $K$ be perfect is replaced by the assumption that $F|K$ be
separable? Note that if $K$ is not perfect, then there exist places $P$
of $F|K$ such that $FP|K$ is not separable, even if $F|K$ is separable.
In this case, the construction of the proof of Theorem~\ref{Fc} breaks
down and it cannot be expected that there is a Frobenius-closed
$K$-basis of $F$ which is as ``natural'' as the ones produced by that
construction.
\end{problem}

\pars
\begin{exercise}
Show that $F/K$ cannot be a free $K[\varphi]$-module if $K$ is not
relatively algebraically closed in $F$. Does there exist an algebraic
field extension which admits a Frobenius-closed basis? Prove a suitable
version of Proposition~\ref{F/Kfree} which does not use the assumption
that $K$ be relatively algebraically closed in $F$.
\end{exercise}

%
%
\section{$G$-modules and group complements}           \label{sectgmgc}
In this section, we introduce some notions that we will need in the next
section. Take any group $G$. For $\sigma\in G$, \bfind{conjugation} {\bf
by $\sigma$} means the automorphism\glossary{$\tau^\sigma$}
\[G\ni \tau\;\mapsto\; \tau^\sigma:= \sigma^{-1} \tau\sigma\;.\]
Note that
\begin{equation}                            \label{G-mod}
\tau^{\sigma\rho}= \rho^{-1}\sigma^{-1}\tau\,\sigma\,\rho =
\rho^{-1}(\tau^\sigma)\rho = (\tau^\sigma)^\rho
\;\;\;\mbox{ for all } \tau,\sigma,\rho\in G\;.
\end{equation}
Further, we set $\tau^{-\sigma}:=(\tau^{-1})^\sigma$ (which indeed is
the inverse of $\tau^\sigma$). As usual, we set $M^\sigma=\{m^\sigma\mid
m\in M\}$ for every subset $M\subset G$. A subgroup $N$ is normal in $G$
if and only if $N^\sigma=N$ for all $\sigma\in G$. We always have
$G^\sigma=G$. Hence, if $H$ is a \bfind{group complement} of the normal
subgroup $N$ in $G$, that is,
\begin{equation}                            \label{HN}
HN=G\;\mbox{\ \ and\ \ }\ H\cap N=\{1\}\;,
\end{equation}
then so is every conjugate $H^\sigma$ for $\sigma\in G$. Uniqueness up to
conjugation would mean that these are the only group complements of $N$
in $G$.

\pars
We shall now introduce two notions that will play an important role in
Section~\ref{sectagfe}. A \bfind{right $G$-module} is an arbitrary group
$N$ together with a mapping $\mu$ from $G$ into the group of automorphisms
of $N$ such that $\mu(\sigma\rho)=\mu(\rho) \circ\mu(\sigma)$. For
example, to every $\sigma\in G$ we may associate the conjugation by
$\sigma$; in view of (\ref{G-mod}), this turns $G$ into a right
$G$-module. In this setting, a subgroup $N$ of $G$ is normal if and only
if it is a $G$-submodule of $G$. A mapping $\phi$ from $G$ into a $G$-module
$N$ is called a \bfind{twisted homomorphism} (or \bfind{crossed
homomorphism}) if it satisfies
\begin{equation}
\phi (\sigma \rho) = \phi (\sigma)^\rho \phi (\rho)\;\;\;\mbox{ for
all }\sigma,\rho\in G\;.
\end{equation}
As for a usual homomorphism, also the kernel of a twisted homomorphism
is a subgroup of $G$, but it may not be normal in $G$.

\pars
Let us assume that $H$ is a group complement of the normal subgroup
$N$ in $G$. It follows from (\ref{HN}) that every element $\sigma\in G$
admits a unique representation
\begin{equation}                            \label{gHgN}\label{sigNsigH}
\sigma=\sigma_H^{ } \sigma_N^{ }\;\;\;\mbox{ with } \; \sigma_H^{ }\in H
\,,\, \sigma_N^{ }\in N
\end{equation}
Note that $H$ is a system of representatives for the left cosets of $G$
modulo $N$. Since $N\lhd G$, we have $HN=NH$, and $H$ is also a system
of representatives for the right cosets of $G$.

Now assume in addition that $N$ is abelian. Then the scalar
multiplication of the $G$-module $N$ given by conjugation reads as
\begin{equation}                            \label{g^h}
\sigma^\rho= \rho_N^{-1}(\rho_H^{-1} \sigma \rho_H^{ }) \rho_N^{ } =
\rho_H^{-1} \sigma \rho_H^{ }=\sigma^{\rho_H^{ }} \;\;\;\mbox{ for all }
\sigma\in N\,,\,\rho\in G
\end{equation}
since $\rho_N^{ }$ and $\rho_H^{-1} \sigma \rho_H^{ }$ are elements
of $N$. According to (\ref{gHgN}) and (\ref{g^h}) we write
\[\sigma \rho=\sigma_H^{ } \sigma_N^{ } \rho_H^{ } \rho_N^{ }=
\sigma_H^{ } \rho_H^{ } \rho_H^{-1}\sigma_N^{ } \rho_H^{ } \rho_N^{
}=\sigma_H^{ } \rho_H^{ } \sigma_N^\rho \rho_N^{ }\;.\]
Hence, the projection $\sigma\mapsto\sigma_H^{ }$ onto the first factor in
(\ref{gHgN}) is the canonical epimorphism from $G$ onto $H$ with
kernel $N$. The other projection $\sigma\mapsto\sigma_N^{ }$ is a twisted
homomorphism from $G$ onto $N$, satisfying
\begin{equation}                            \label{twhom}
(\sigma \rho)_N^{ } = \sigma_N^\rho \rho_N^{ } \;\;\;\mbox{ for all }
\sigma,\rho\in G\;;
\end{equation}
it induces the identity on $N$, and its kernel is $H$.

%
%
\section{Field extensions generated by $p$-polynomials} \label{sectagfe}
{\bf In this section, let $K$ be a field of characteristic $p>0$.}
By a \bfind{Galois extension} we mean a normal and separable, but not
necessarily finite algebraic extension. A field extension $L|K$ is
called \bfind{$p$-elementary extension} if it is a finite Galois
extension and its Galois group is an elementary-abelian $p$-group, that
is, an abelian $p$-group in which every nonzero element has order $p$.
In particular, $[L:K]$ is a power of $p$.

In this section, we will consider the following larger class of all
extensions $L|K$ which satisfy the following condition:
\begin{equation}                            \label{semi-p-elcond}
\left.\begin{array}{l}
\mbox{there exists a Galois extension $K'|K$ which is linearly
disjoint from $L|K$,}\\
\mbox{such that $L.K'|K'$ is a $p$-elementary extension}\\
\mbox{and also $L.K'|K$ is a Galois extension.}
\end{array}\right\}
\end{equation}
From the linear disjointness it follows that $\Gal L.K'|L \isom \Gal
K'|K$ and that $[L:K]=[L.K':K']$ which yields that $[L:K]=p^n$ for some
natural number $n$. For a further investigation of this situation, we
will use the following notation. We set
\[L'\>:=\>L.K'\]
and define
\sn
$\bullet\quad G:=\Gal L'|K\,$,\n
$\bullet\quad N:=\Gal L'|K'\lhd G\,$,\n
$\bullet\quad H:=\Gal L'|L \isom \Gal K'|K \isom G/N\,$.
\sn
The group $N$ is abelian of order $p^n$. Since $K'|K$ is assumed to be a
Galois extension, $N$ is a normal subgroup of $G$. That is, $N$ is a
\ind{right $G$-module} with scalar multiplication given by conjugation:
\[(\sigma,\tau) \;\mapsto\; \sigma^\tau = \tau^{-1}\sigma\tau
\;\;\;\mbox{ for all } \sigma\in N\,,\,\tau\in G\;.\]
Since $L.K'=L'$ and $L\cap K'=K$, we have that $H\cap N=1$ and $G=HN$,
that is, $H$ is a group complement for $N$ in $G$. As we have seen in
the last section, every element $\sigma\in G$ admits a unique
representation (\ref{gHgN}). Since $N$ is abelian, the scalar
multiplication of the $G$-module $N$ is given by (\ref{g^h}). The
projection $\sigma\mapsto\sigma_N^{ }$ is a twisted homomorphism from
$G$ onto $N$, satisfying (\ref{twhom}); it induces the identity on $N$,
and its kernel is $H$.

\begin{lemma}
If $L|K$ satisfies condition (\ref{semi-p-elcond}) then w.l.o.g., the
extension $K'|K$ may assumed to be finite (which yields that also $L'|K$
is finite).
\end{lemma}
\begin{proof}
Suppose that $K'|K$ is an arbitrary algebraic extension such that
(\ref{semi-p-elcond}) holds. Let $N_0:=\{\tau\in G\mid \forall\sigma_N^{ }
\in N:\;\tau^{-1} \sigma_N^{ }\tau= \sigma_N^{ }\}$ be the subgroup of all
automorphisms in $G$ whose action on $N$ is trivial (i.e., $N_0$ is the
centralizer of $N$ in $G$). Since $N$ is abelian, it is contained in
$N_0\,$. Consequently, the fixed field $K_0$ of $N_0$ in $L'$ is
contained in $K'$ (which is the fixed field of $N$ in $L'$ by definition
of $N$). Since $N$ is a normal subgroup of $G$, also its centralizer
$N_0$ is a normal subgroup of $G$, showing that $K_0|K$ is a Galois
extension. We set $H_0:= G/N_0$. By our choice of $N_0\,$, the action of
$G$ on $N$ induces an action of $H_0$ on $N$ which is given by
$\rho^{-1} \sigma_N^{ } \rho = \tau^{-1}\sigma_N^{ }\tau$ for $\rho=\tau
N_0\in H_0\,$. Consequently, $H_0$ must be finite, being a group of
automorphisms of the finite group $N$. This proves that $K_0|K$ is a
finite Galois extension with Galois group $H_0\,$. Recall that it
follows from (\ref{semi-p-elcond}) that also $L|K$ is finite.

We claim that $H\cap N_0$ is a normal subgroup of $G$. Let $\tau\in H
\cap N_0$ and $\sigma\in G$; we want to show that $\tau^\sigma\in H\cap
N_0$. Write $\sigma=\sigma_H^{ }\sigma_N^{ }$ according to
(\ref{sigNsigH}). Then $\tau^\sigma=\sigma_N^{-1} (\sigma_H^{-1}
\tau\sigma_H^{ })\sigma_N^{ }\,$; since $\sigma_H^{ }\in H$ and $N_0\lhd
G$, we find that $\tau':= \sigma_H^{-1}\tau\sigma_H^{ }\in H\cap N_0$.
In particular, $\tau'$ lies in the centralizer of $N$. In view of
$\sigma_N^{ }\in N$ we obtain $\tau^\sigma = \sigma_N^{-1}\tau'
\sigma_N^{ } = \sigma_N^{-1}\sigma_N^{ }\tau' =\tau'\in H\cap N_0$. We
have proved that $H\cap N_0$ is a normal
subgroup of $G$. With $L_0$ the fixed field of $H\cap N_0$ in $L'$, we
hence obtain a Galois extension $L_0|K$. Since $L_0=L.K_0$, the
extension $L_0|K$ is finite.

Finally, it remains to show that $\Gal L_0|K_0 \isom \Gal L'|K'$
which also yields that $L_0|K_0$ is $p$-elementary. Observe that
$HN_0=G$ since it contains $HN=G$. Now we compute:
$\Gal L_0|K_0= \Gal L'|K_0\,/\,\Gal L'|L_0 = N_0/(H\cap N_0)\isom
H.N_0/H =G/H\isom N = \Gal L'|K'$. We have proved that condition
(\ref{semi-p-elcond}) also holds with $K_0\,,\,L_0$ in the place of
$K'\,,\,L'$.
\end{proof}
\n
{\bf In view of this lemma, we will assume in the sequel that all field
extensions are finite.}

\pars
Like $N$, also the additive group $(L',+)$ is a right $G$-module,
the scalar multiplication given by
\[(a,\tau)\;\mapsto\; a^\tau:=\tau^{-1} a
\;\;\;\mbox{ for all } a\in L'\,,\,\tau\in G\;.\]
Let us show:
\begin{lemma}
There is an embedding $\phi:\; N \longrightarrow (L',+)$
of right $G$-modules.
\end{lemma}
\begin{proof}
By the Normal Basis Theorem (cf.\ \fvklit{L}), the finite Galois
extension $K'|K$ admits a normal basis. That is, there exists $b\in K'$
such that $b^\rho$, $\rho \in \Gal K'|K\,$, is a basis of $K'$ over $K$.
Since $H$ is a set of representatives in $G$ for $\Gal K'|K$, we may
represent these conjugates as $b^\rho$, $\rho \in H$. Let $\psi:\; N
\rightarrow (K,+)$ be any homomorphism of groups (there is always at
least the trivial one), and set
\begin{equation}                            \label{Phi}
\phi(\sigma_N^{ }) := \sum_{\rho\in H} \psi(\rho\sigma_N^{ }\rho^{-1})\,
b^\rho\;\;\; \mbox{ for all }\sigma_N^{ } \in N\;.
\end{equation}
Since $\psi$ is a group homomorphism from $N$ into $(L',+)$, the same is
true for $\phi$. Given $\tau\in G$, we write $\tau=\tau_H^{ }
\tau_N^{ }\,$; then $b^{\rho\tau}= (b^{\rho\tau_H^{ }})^{\tau_N^{ }}=
b^{\rho\tau_H^{ }}$ since $b^{\rho\tau_H^{ }}\in K'$ and $\tau_N^{ }\in
N=\Gal L'|K'$. Observing also that $H=H\tau_H^{ }$ and using
(\ref{g^h}), we compute:

\begin{eqnarray*}
\phi(\sigma_N^{ })^\tau & = & \sum_{\rho\in H} \psi(\rho\sigma_N^{ }
\rho^{-1}) \, b^{\rho\tau}\; =\> \sum_{\rho\in H} \psi(\rho\tau_H^{-1}
\sigma_N^{ }(\rho\tau_H^{-1})^{-1})\, b^\rho\\
& = & \sum_{\rho\in H} \psi(\rho\sigma_N^{\,\tau_H^{ }}\rho^{-1})\,
b^\rho \;=\> \phi(\sigma_N^{\,\tau_H^{ }}) \;=\> \phi(\sigma_N^\tau)
\end{eqnarray*}
which shows that $\phi$ is a homomorphism of right $G$-modules.

Now we have to choose $\psi$ so well as to guarantee that $\phi$ becomes
injective. If $\phi(\sigma_N^{ })=0$ then $\psi(\rho\sigma_N^{ }
\rho^{-1})=0$ for all $\rho\in H$ since by our choice of $b$, the
conjugates $b^\rho$, $\rho\in H$, are linearly independent over $K$. In
particular, $\phi(\sigma_N^{ })=0$ implies $\psi(\sigma_N^{ })=0$.
Hence, $\phi$ will be injective if we are able to choose $\psi$ to be
injective. This is done
as follows. The elementary-abelian $p$-group $N$ may be viewed as a
finite-dimensional $\Fp$-vector space. If $K$ is an infinite field
(which by our general assumption has characteristic $p$), then it
contains $\Fp$-vector spaces of arbitrary finite dimension; so there
exists an embedding $\psi$ of $N$ into $(K,+)$. If $K$ is a finite
field, then all finite extensions of $K$ are cyclic, their Galois groups
being generated by a suitable power of the Frobenius $\varphi$;
consequently, $N$ must be cyclic. Since it is also elementary-abelian,
$N$ is isomorphic to $\Z/p\Z$ which is the additive group of $\Fp\subset
K$. Hence also in this case, $N$ admits an embedding $\psi$ into
$(K,+)$.
\end{proof}

By composition with $\phi$, the twisted homomorphism $\sigma\mapsto
\sigma_N^{ }$ is turned into a mapping $\sigma\mapsto\phi (\sigma_N^{ })$
from $G$ into $(L',+)$. We shall write $\phi(\sigma)$ instead of $\phi
(\sigma_N^{ })$, thereby considering the $G$-module homomorphism $\phi:\;
N\rightarrow (L',+)$ as being extended to $\phi:\; G\rightarrow (L',+)$.
By construction, the latter has kernel $H$ and is injective on $N$.
Further, it satisfies $\phi(\sigma\tau)=\phi((\sigma\tau)_N^{ })=
\phi(\sigma_N^\tau \tau_N^{ }) = \phi(\sigma_N^{ })^\tau+
\phi(\tau_N^{ })= \phi(\sigma)^\tau + \phi(\tau)$ showing that $\phi$ is
a twisted homomorphism in the following sense:
\begin{equation}
\phi(\sigma\tau)=\phi(\sigma)^\tau +\phi(\tau)\;\;\;\mbox{ for all }
\sigma,\tau\in G\;.
\end{equation}

\pars
We claim that there exists an element $\vartheta\in L'$ such that
\begin{equation}                            \label{sigmatheta}
\vartheta^\tau=\vartheta+\phi(\tau)\;\;\; \mbox{ for all }\tau\in G\;.
\end{equation}
Note that (\ref{sigmatheta}) determines $\vartheta$ up to addition of
elements from $K$. (Indeed, $\vartheta'$ satisfies the same equation
if and only if $(\vartheta-\vartheta')^\tau=\vartheta-\vartheta'$, i.e.,
if and only if $\vartheta-\vartheta'\in K$.)

The element $\vartheta$ can be constructed as follows. We choose an
element $a\in L'$ such that the trace $s:=\mbox{\rm Tr}_{L'|K} (a)=
\sum_{\sigma\in G} \sigma a=\sum_{\sigma\in G} a^\sigma$ is not zero
(we have seen in the foregoing proof that such an element exists: we
could choose $a$ to be the generator of a normal basis of $L'|K$; the
linear independence will then force the trace to be nonzero). We set
\begin{equation}
\vartheta:= - \frac{1}{s}\sum_{\sigma\in G} \phi(\sigma)\, a^\sigma\;.
\end{equation}
Given $\tau\in G$, we have $G\tau=G$ and
\[1=\frac{1}{s}\sum_{\sigma\in G} a^\sigma =
\frac{1}{s}\sum_{\sigma\in G} a^{\sigma\tau}\]
which we use to compute
\begin{eqnarray*}
\vartheta^\tau & = & - \frac{1}{s}\sum_{\sigma\in G} \phi(\sigma)^\tau
a^{\sigma\tau} \;=\; -\frac{1}{s}\sum_{\sigma\in G}\left((
\phi(\sigma)^\tau +\phi(\tau))a^{\sigma\tau}
-\phi(\tau)a^{\sigma\tau}\right)\\
& = & - \frac{1}{s}\sum_{\sigma\in G} \phi(\sigma\tau)a^{\sigma\tau}
+ \phi(\tau)\frac{1}{s}\sum_{\sigma\in G} a^{\sigma\tau}\\
& = & - \frac{1}{s}\sum_{\sigma\in G} \phi(\sigma)a^\sigma
+ \phi(\tau)\frac{1}{s}\sum_{\sigma\in G} a^\sigma
\;=\; \vartheta + \phi(\tau)\;.
\end{eqnarray*}
This proves that $\vartheta$ indeed satisfies (\ref{sigmatheta}).
\begin{remark}
The additive analogue of Hilbert's Satz 90 (cf.\ \fvklit{L} or
\fvklit{J}, chapter 1, section 15) says that ${\rm H}^1 (G,(L',+))=0$.
Since the twisted homomorphism $\phi:\; G\rightarrow (L',+)$ may be
interpreted as a $1$-cocycle, this implies that $\phi$ splits, which
indicates the existence of $\vartheta$. Replacing the twisted
homomorphism $\phi$ by an arbitrary $1$-cocycle in our above computation
provides a proof of this additive analogue.
\end{remark}

Since $H$ is the kernel of $\phi$, (\ref{sigmatheta}) yields that $H$ is
the group of all automorphisms of $L'|K$ which fix $\vartheta$. Since on
the other hand, by definition of $H=\Gal L'|L$ the fixed field of $H$ in
$L'$ is $L$, we know from Galois theory that $L=K(\vartheta)$. Let us
now compute the minimal polynomial $f$ of $\vartheta$ over $K$. The
group $N$ may be viewed as a system of representatives for the left
cosets of $G$ modulo $H$. Consequently, the elements $\vartheta^\tau$,
$\tau\in N$, are precisely all conjugates of $\vartheta$ over $K$. So
\[f(X)=\prod_{\tau\in N}(X-\vartheta^\tau)=
\prod_{\tau\in N}(X-\vartheta-\phi(\tau))={\cal A}(X-\vartheta)\;,\]
where
\[{\cal A}(X):=\prod_{\tau\in N}(X-\phi(\tau))\;.\]
The roots of ${\cal A}$ form the additive group $\phi(N)$. Since
we have chosen $\phi$ to be injective, we have $|\phi(N)|=|N|=\deg
{\cal A}$. By part a) of Corollary~\ref{charaddpol} it follows that
${\cal A}$ is an additive polynomial. In particular,
\[f(X)={\cal A}(X-\vartheta)={\cal A}(X)-{\cal A}(\vartheta)\;.\]
Since $f(X)\in K[X]$, we have ${\cal A}(X)\in K[X]$ and ${\cal A}
(\vartheta)\in K$. Since $\deg f=\deg {\cal A}=|N|=[L:K]=[K(\vartheta)
:K]$, $f$ is the minimal polynomial of $\vartheta$ over $K$.

We have proved:
\begin{theorem}                               \label{semi-p-el}
Let $L|K$ be an extension which satisfies condition
(\ref{semi-p-elcond}). Then there exist an additive polynomial
${\cal A}(X)\in K[X]$ and an element $\vartheta\in L$ such
that $L=K(\vartheta)$ and ${\cal A}(X)-{\cal A}(\vartheta)\in K[X]$ is
the minimal polynomial of $\vartheta$ over $K$.
\end{theorem}

As an example, let us discuss an important special case. Let us assume
that $L|K$ is a Galois extension of degree $p$. Then its Galois group is
just $\Z/p\Z$, and the extension is thus $p$-elementary. In the above
setting, we may then choose $K'=K$ which yields $L'=L$, $G=N=\Z/p\Z$ and
$H=1$. The embedding $\phi:\; N \longrightarrow (L',+)$ may be chosen
``by hand'' to be the most natural one: $N=\Z/p\Z=(\Fp,+)\subset
(L',+)$. We obtain
\[{\cal A}(X)=\prod_{i\in \Fp}(X-i)=X^p - X\]
since the latter is the unique polynomial of degree $p$ which vanishes
on all elements of $\Fp$. The extension $L|K$ is thus generated by the
root $\vartheta$ of the polynomial $f(X)=X^p - X-{\cal A}(\vartheta)$
which we call an \bfind{Artin-Schreier polynomial}. The extension $L|K$
is an Artin-Schreier extension. So we have shown:
\begin{theorem}                           \label{GalpAS}
Every Galois extension of degree $p$ of a field of characteristic $p>0$
is an Artin-Schreier extension.
\end{theorem}

\parm
Inspired by this special case, we want to investigate whether we can get
more information about the additive polynomial ${\cal A}$ if we
strengthen the hypotheses. For instance, $\phi$ may be injective even if
$\psi$ is not. In our special case, $N=\Z/p\Z$ was an irreducible
$G$-module, that is, it did not admit any proper nonzero $G$-submodule.
But if $N$ is an irreducible $G$-module, then every $G$-module
homomorphism $\phi$ can only have kernel $0$ or $N$, so if it does not
vanish, then it is injective. For $\phi$ as defined in (\ref{Phi}), we
obtain $\phi\not=0$ if $\psi\not=0$. So it will suffice to take $\psi:\;
N\rightarrow (\Fp\,,+)$ as a nonzero (additive) character; it exists
since $N$ is a non-trivial $p$-group. With this choice of $\psi$, we
obtain
\[\phi(N)\subset \sum_{\rho\in H} \Fp\, b^\rho =
\sum_{\rho\in H} \Fp\, \rho b\;.\]
Since the coefficients of the polynomial ${\cal A}$ are the elementary
symmetric polynomials of the elements $\phi(\tau)$, $\tau\in N$, they
lie in the ring $\Fp[\,\rho b\mid\rho\in H]$.

The condition that $N$ be an irreducible $G$-module has turned out to be
of certain importance. It is satisfied in the following special case:
\begin{lemma}                               \label{minGmod}
Assume that $L|K$ is minimal with the property (\ref{semi-p-elcond}),
that is, there is no proper non-trivial subextension with the same
property. Then $N$ is an irreducible $G$-module.
\end{lemma}
\begin{proof}
Assume that $M$ is a $G$-submodule of $N$, that is, $M$ is a normal
subgroup of $G$. Then $HM$ is a subgroup of $G$ containing $H$.
In view of the unique representation (\ref{sigNsigH}), we have $HM=H$ if
and only if $M=1$ and $HM=G$ if and only if $M=N$. Note that the fixed
field $L'_1$ of $M$ in $L'$ is a Galois extension of $K$ containing
$K'$. Further, the fixed field $L_1$ of $HM$ is contained in $L$, and it
satisfies $L_1.K'=L'_1$ since $HM\cap N=M\cap N = M$. Consequently, also
$L_1|K$ has property (\ref{semi-p-elcond}).

Suppose now that $L|K$ is minimal with the property
(\ref{semi-p-elcond}). Then $L_1=L$ or $L_1=K$. Hence $HM=H$ or $HM=G$,
that is, $M=1$ or $M=N$, showing that the $G$-module $N$ is irreducible.
\end{proof}

We summarize our preceding discussion in the following
\begin{lemma}                               \label{semi-p-elzus}
Assume that $L|K$ is minimal with the property (\ref{semi-p-elcond}). If
$K'|K$ is infinite, we may replace it by a suitable finite subextension.
For every $b\in K'$ generating a normal basis of $K'|K$, and for every
nonzero additive character $\psi:\; N\rightarrow (\Fp,+)$, the
$G$-module homomorphism $\phi$ defined in (\ref{Phi}) is injective.
Moreover, the coefficients of the correponding additive polynomial
${\cal A}(X)$ lie in the ring
\[K\cap \Fp[\,\rho b\mid\rho\in H]\;.\]
\end{lemma}

\mn
\begin{exercise}
Let $\chara K=p>0$ and $L|K$ be an Artin-Schreier extension and
$\vartheta$ an \bfind{Artin-Schreier generator} of $L|K$, that is,
$L=K(\vartheta)$ and $\vartheta^p-\vartheta\in K$. Show that all other
Artin-Schreier generators of $L|K$ are of the form $i\vartheta+c$ with
$i\in\{1,2,\ldots,p-1\}$ and $c\in K$. Can something similar be said in
the setting of Theorem~\ref{semi-p-el}? (Hint: use the uniqueness
statement following equation (\ref{sigmatheta})).
\end{exercise}

%
%
\section{Minimal purely wild extensions}    \label{sectmpwe}
This section is devoted to the proof of Theorem~\ref{mpwe} which shows
the important connection between purely wild (and in particular,
immediate) extensions of henselian fields of positive characteristic and
additive polynomials.

Before we continue, we put together several facts from ramification
theory that can be found in [En], [N] and [Ku12] or can be deduced
easily from other facts (exercise for the reader). For a field $L$, we
denote by $\tilde{L}$ its algebraic closure and by $\Gal L$
the absolute Galois group $\Gal \tilde{L}|L$ of $L$. Recall that for a
henselian field $(K,v)$, $K^r$ denotes the ramification field of the
extension $(K\sep|K,v)$.

\begin{theorem}                             \label{ram}
Let $(K,v)$ be a henselian field and $p$ the characteristic exponent of
$Kv$. Then the following assertions hold:
\sn
a) $\Gal K^r$ is a normal subgroup of $\Gal K$ and $K^r|K$ is a Galois
extension.
\n
b) $\Gal K^r$ is a pro-$p$-group, so the separable-algebraic closure of
$K$ is a $p$-extension of $K^r$.
\n
c) The value group $vK^r$ consists of all elements in the ordered
divisible hull of $vK$ whose order modulo $vK$ is prime to $p$. The
residue field $K^rv$ is the separable-algebraic closure of $Kv$.
\n
d) If $vK^r=vK$ (we say that $K^r|K$ is \bfind{unramified}), then
for every Galois subextension $K'|K$ of $K^r|K$, we have $\Gal K'|K
\isom\Gal K'v|Kv$.
\n
e) Every finite extension $(K_2|K_1,v)$, where $K\subseteq K_1\subseteq
K_2\subseteq K^r$, is defectless.
\n
f) If $L$ is an algebraic extension of $K$, then $L^r=L.K^r$, and the
extensions $(L^r|K^r,v)$ and  $(L|K,v)$ have the same defect.
\n
g) If $(L,v)$ is an immediate henselian extension of $(K,v)$, not
necessarily algebraic, then $L^r=L.K^r$.
\end{theorem}

Let $(K,v)$ be a henselian field which is not tame and thus admits
purely wild extensions (see Theorem~\ref{pank}). Theorem~\ref{mpwe} will
follow from Theorem~\ref{semi-p-el} if we are able to show that every
minimal purely wild extension $L|K$ satisfies condition
(\ref{semi-p-elcond}) which is the hypothesis the latter theorem. As a
natural candidate for an extension $K'|K$ which is Galois and linearly
disjoint from $L|K$, we can take the extension $K^r|K$. By part f) of
Theorem~\ref{ram} we know that $L.K^r=L^r$. We set
\sn
$\bullet\quad {\cal G}:=\Gal K$,\n
$\bullet\quad {\cal N}:=\Gal K^r$, which is a normal subgroup of
${\cal G}$ and a pro-$p$-group,\n
$\bullet\quad {\cal H}:=\Gal L$, which is a maximal proper subgroup of
${\cal G}$ since $L|K$ is a minimal non-trivial extension, and which
satisfies ${\cal N}.{\cal H}={\cal G}$ since $K^r|K$ is linearly
disjoint from $L|K$\n
$\bullet\quad {\cal D}:={\cal N}\cap {\cal H}= \Gal L.K^r =\Gal L^r\;$.
\sn
The next lemma examines this group theoretical situation.
\begin{lemma}
Let ${\cal G}$ be a profinite group with maximal proper subgroup
${\cal H}$. Assume that the non-trivial pro-$p$-group ${\cal N}$ is a
normal subgroup of ${\cal G}$ not contained in ${\cal H}$. Then
${\cal D} ={\cal N}\cap {\cal H}$ is a normal subgroup of ${\cal G}$ and
the finite factor group ${\cal N}/{\cal D}$ is an elementary-abelian
$p$-group. Further, ${\cal N}/{\cal D}$ is an irreducible right
${\cal G}/{\cal D}$-module.
\end{lemma}
\begin{proof}
By the maximality of ${\cal H}$, we have ${\cal H} {\cal N}={\cal G}$.
Since ${\cal N}\not\subset {\cal H}$, we have that ${\cal D}$ is a
proper subgroup of ${\cal N}$. Since every maximal proper subgroup of a
profinite group is of finite index, we have that $({\cal N}:{\cal D})=
({\cal G}:{\cal H})$ is finite. Observe that ${\cal D}$ is
\bfind{${\cal H}$-invariant} (which means that ${\cal D}^\sigma=
{\cal D}$ for every $\sigma\in {\cal H}$). This is true since ${\cal N}
\lhd {\cal G}$ and ${\cal H}$ are ${\cal H}$-invariant. Assume that
${\cal E}$ is an ${\cal H}$-invariant subgroup of ${\cal N}$
containing ${\cal D}$. Then ${\cal H} {\cal E}$ is a subgroup of
${\cal G}$ containing ${\cal H}$. From the maximality of ${\cal H}$ it
follows that either ${\cal H}{\cal E}={\cal H}$ or ${\cal H}{\cal E}
={\cal G}$, whence either ${\cal E}={\cal D}$ or ${\cal E}={\cal N}$
(this argument is as in the proof of Lemma~\ref{minGmod}). We have
proved that ${\cal D}$ is a maximal ${\cal H}$-invariant subgroup of
${\cal N}$.

Now let $\Phi({\cal N})$ be the Frattini subgroup of ${\cal N}$, i.e.,
the intersection of all maximal open subgroups of ${\cal N}$. Since
${\cal D}\not={\cal N}$, we can pick some maximal
proper subgroup of ${\cal N}$ containing ${\cal D}$, and since it also
contains $\Phi({\cal N})$, it follows that ${\cal D}\Phi({\cal N})
\not= {\cal N}$. Being a characteristic subgroup of ${\cal N}$, the
Frattini subgroup $\Phi({\cal N})$ is ${\cal H}$-invariant like
${\cal N}$. Consequently, also the group ${\cal D}\Phi({\cal N})$ is
${\cal H}$-invariant. From the maximality of ${\cal D}$ we deduce that
${\cal D}\Phi({\cal N})={\cal D}$, showing that
\[
\Phi({\cal N})\subset {\cal D}\;.
\]
On the other hand, the factor
group ${\cal N}/\Phi({\cal N})$ is a (possibly infinite dimensional)
$\Fp$-vector space (cf.\ [R--Zal], part (b) of Lemma~2.8.7). In view of
$\Phi({\cal N})\subset {\cal D}$, this yields that also ${\cal D}$ is a
normal subgroup of ${\cal N}$ and that also ${\cal N}/{\cal D}$ is an
elementary-abelian $p$-group. Since ${\cal D}$ is ${\cal H}$-invariant,
${\cal D}\lhd {\cal N}$ implies that
\[{\cal D}\lhd {\cal H} {\cal N}= {\cal G}\;.\]
As a normal subgroup of ${\cal G}$, ${\cal N}$ is a ${\cal G}$-module,
and in view of ${\cal D}\lhd {\cal G}$ it follows that ${\cal N}/{\cal
D}$ is a ${\cal G}/{\cal D}$-module. If it were reducible then there
would exist a proper subgroup ${\cal E}$ of ${\cal N}$ such that
${\cal E}/{\cal D}$ is a non-trivial ${\cal G}/{\cal D}$-module. But
then, ${\cal E}$ must be a normal subgroup of ${\cal G}$ properly
containing ${\cal D}$; in particular, ${\cal E}$ would be a proper
${\cal H}$-invariant subgroup of ${\cal N}$, in contradiction to the
maximality of ${\cal D}$.
\end{proof}
\n
This lemma shows that $L^r|K$ is Galois and $L^r=L.K^r$ is a finite
$p$-elementary extension of $K^r$. Hence $L|K$ satisfies
(\ref{semi-p-elcond}) with $K'=K^r$. We apply Theorem~\ref{semi-p-el} to
obtain an additive polynomial ${\cal A}(X)\in K[X]$ and an element
$\vartheta\in L$ such that $L=K(\vartheta)$ and that ${\cal A}(X)-
{\cal A} (\vartheta)$ is the minimal polynomial of $\vartheta$ over $K$.
Since $L|K$ is a minimal purely wild extension by our assumption, it is
in particular minimal with property (\ref{semi-p-elcond}) and thus
satisfies the hypothesis of Lemma~\ref{semi-p-elzus}. Hence, the
extension $K^r|K$ can be replaced by a finite subextension $K'|K$, and
${\cal A} (X)$ may be chosen such that its coefficients lie in the ring
$K\cap \Fp[\,\rho b\mid\rho \in H]$, where $b$ is the generator of a
normal basis of $K'|K$. Since $vK$ is cofinal in
$v\tilde{K}=\widetilde{vK}$, we may choose some $c\in K$ such that
$vcb\geq 0$. Since $(K,v)$ is henselian by assumption, it follows that
$v\sigma (cb)=vcb\geq 0$ for all $\sigma \in \Gal K$. On the other hand,
$cb$ is still the generator of a normal basis of $K'|K$. So we may
replace $b$ by $cb$, which yields that $K\cap \Fp[\,\rho b\mid \rho\in
H]\subset {\cal O}_K$ and consequently, that ${\cal A} (X)\in
{\cal O}_K[X]$.

\pars
Now assume in addition that $(K|k,v)$ is an immediate extension of
henselian fields. Then we may infer from part g) of Theorem~\ref{ram}
that $K^r=k^r.K\,$. So the Galois extension $K'$ of $K$ is the
compositum of $K$ with a suitable Galois extension $k'$ of $k$. In this
case, $b$ may be chosen to be already the generator of a normal basis of
$k'$ over $k$; it will then also be the generator of a normal basis of
$K'$ over $K$. With this choice of $b$, we obtain that the ring $K\cap
\Fp[\,\rho b\mid \rho\in H]$ is contained in $K\cap k'=k$, whence ${\cal
A} (X)\in {\cal O}_k[X]$. Let us summarize what we have proved; the
following theorem will imply Theorem~\ref{mpwe}.

\begin{theorem}
Let $(K,v)$ be a henselian field and $(L|K,v)$ a minimal purely wild
extension. Then $L^r|K$ is a Galois extension and $L^r|K^r$ is a
$p$-elementary extension. Hence, $L|K$ satisfies condition
(\ref{semi-p-elcond}), and there exist an additive polynomial ${\cal A}
(X)\in {\cal O}_K[X]$ and an element $\vartheta\in L$ such that $L=
K(\vartheta)$ and that ${\cal A}(X)-{\cal A} (\vartheta)$ is the minimal
polynomial of $\vartheta$ over $K$. If $(K|k,v)$ is an immediate
extension of henselian fields, then ${\cal A}(X)$ may already be chosen
in ${\cal O}_k[X]$.
\end{theorem}

Let us conclude this section by discussing the following special case.
Assume that the value group $vK$ is divisible by all primes $q\not=p$.
Then by part c) of Theorem~\ref{ram}, $K^r|K$ is an unramified
extension. Consequently by part d) of Theorem~\ref{ram}, $\Gal K'|K
\isom\Gal \ovl{K'}|\ovl{K}$ and we may choose the element $b$ such that
$\ovl{b}$ is the generator of a normal basis of $\ovl{K'}|\ovl{K}$. It
follows that the residue mapping is injective on the ring $\Fp[\rho
b\mid\rho\in H]$ and thus, also the mapping $\tau\mapsto \ovl{\phi(\tau)}$
is injective. In this case, we obtain that
\[\ovl{\cal A}(X)=\prod_{\tau\in N}(X-\ovl{\phi(\tau)})\]
has no multiple roots and is thus separable.

%
%
\section{$p$-closed fields}                 \label{sectp-cl}
This section is devoted to the proofs of the two theorems that deal with
$p$-closed fields of positive characteristic.
\sn
{\bf Proof of Theorem~\ref{Whaples}:} \ We have to prove that\sn
{\it A field $K$ is $p$-closed if and only if it does not admit any
finite extensions of degree divisible by $p$.}
\sn
``$\Leftarrow$'': \ Assume that $K$ does not admit any finite extensions
of degree divisible by $p$. Take any $p$-polynomial $f\in K[X]$. Write
$f={\cal A}+c$ where ${\cal A}\in K[X]$ is an additive polynomial. Let
$h$ be an irreducible factor of $f$; by hypothesis, it has a degree $d$
not divisible by $p$. Fix a root $b$ of $h$ in the algebraic closure
$\tilde{K}$ of $K$. All roots of $f$ are of the form $b+a_i$ where the
$a_i$s are roots of ${\cal A}$. By part a) of Corollary~\ref{charaddpol}
the roots of ${\cal A}$ in $\tilde{K}$ form an additive group. The sum
of the roots of $h$ lies in $K$. This gives us $db+s\in K$, where $s$ is
a sum of a subset of the $a_i$s and is therefore again a root of
${\cal A}$. Likewise, $d^{-1}s$ is a root of ${\cal A}$ (as $d$ is not
divisible by $p$, it is invertible in $K$). Then $b+d^{-1}s=d^{-1}
(db+s)$ is a root of $f$, and it lies in $K$, as required.
\sn
``$\Rightarrow$'': \ (This part of the proof is due to David
Leep.) Assume that $K$ is $p$-closed. Since $K$ is
perfect, it suffices to take a Galois extension $L|K$ of degree $n$ and
show that $p$ does not divide $n$. By the normal basis theorem there is
a basis $b_1,\ldots,b_n$ of $L|K$ where the $b_i$s are the roots of some
irreducible polynomial over $K$. Since they are linearly independent
over $K$, their trace is non-zero. The elements
\[1,b_1,b_1^p,\ldots,b_1^{p^{n-1}}\]
are linearly dependent over $K$ since $[L:K]=n$. Therefore there exist
elements $d_0,\ldots,d_{n-1},e\in K$ such that the $p$-polynomial
\[f(X)\>=\>d_{n-1}X^{p^{n-1}}+\ldots+d_0X+e\]
has $b_1$ as a root. It follows that all the $b_i$s are roots of $f$.
Thus the elements $b_2-b_1,\ldots,b_n-b_1$ are roots of the additive
polynomial $f(X)-e$. Since these $n-1$ roots are linearly independent
over $K$, they are also linearly independent over the prime field
$\F_p\,$. This implies that the additive group $G$ generated by the
elements $b_2-b_1,\ldots,b_n-b_1$ contains $p^{n-1}$ distinct elements,
which therefore must be precisely the roots of $f(X)-e$. So $G+b_1$ is
the set of roots of $f$. By hypothesis, one of these roots lies in $K$;
call it $\vartheta$. There exist integers $m_2,\ldots,m_n$ such
that
\[\vartheta\>=\>m_2(b_2-b_1)+\ldots+m_n(b_n-b_1)+b_1\,.\]
In this equation take the trace from $L$ to $K$. The elements
$b_1,\ldots,b_n$ all have the same trace; hence the trace of every
$m_i(b_i-b_1)$ is $0$. It follows that the trace $n\vartheta$ of
$\vartheta$ is equal to the trace of $b_1\,$; as we have remarked
already, this trace is non-zero. Hence $n\vartheta\ne 0$, which shows
that $n$ is not divisible by $p$.                           \QED

\mn
{\bf Proof of Theorem~\ref{amp-cl=Kap}:} \ We have to prove:\sn
{\it A henselian valued field of characteristic $p>0$ is $p$-closed if
and only if it is an algebraically maximal Kaplansky field.}
\sn
We will use Theorem~\ref{Whaples} throughout the proof without further
mention. Assume first that $(K,v)$ is henselian and that $K$ is
$p$-closed. Since every finite extension of the residue field $Kv$ can
be lifted to an extension of $K$ of the same degree, it follows that
$Kv$ is $p$-closed. Likewise, if the value group $vK$ were not
$p$-divisible, then $K$ would admit an extension of degree $p$; this
shows that $vK$ is $p$-divisible. We have thus proved that $(K,v)$ is a
Kaplansky field. Since the degree of every finite extension of $K$ is
prime to $p$, it follows that $(K,v)$ is defectless, hence
algebraically maximal.

For the converse, assume that $(K,v)$ is an algebraically maximal
Kaplansky field. Since the henselization is an immediate algebraic
extension, it follows that $(K,v)$ is henselian. By Theorem~\ref{pank},
there exists a field complement $W$ of $K^r$ in $\tilde{K}$. As $vK$ is
$p$-divisible and $Kv$ is $p$-closed, hence perfect, the same theorem
shows that $W$ is an immediate extension of $K$. Hence $W=K$, which
shows that $K^r= \tilde{K}$. So every finite extension $L|K$ is a
subextension of $K^r|K$ and is therefore defectless; that is,
$[L:K]=(vL:vK)[Lv:Kv]$. As the right hand side is not divisible by $p$,
$(K,v)$ being a Kaplansky field, we find that $p$ does not divide
$[L:K]$. By Theorem~\ref{Whaples}, this proves that $K$ is $p$-closed.
\QED

\bn
\bn
{\bf References}
\newenvironment{reference}%
{\begin{list}{}{\setlength{\labelwidth}{6em}\setlength{\labelsep}{0em}%
\setlength{\leftmargin}{6em}\setlength{\itemsep}{-1pt}%
\setlength{\baselineskip}{3pt}}}%
{\end{list}}
\newcommand{\lit}[1]{\item[{#1}\hfill]}
\begin{reference}
\lit{[B]} {Bourbaki, N.$\,$: {\it Commutative algebra}, Paris (1972)}
\lit{[C1]} {Cohn, P.~M.$\,$: {\it Free rings and their relations},
London Math.\ Soc.\ Monograph {\bf 2}, London (1971)}
\lit{[C2]} {Cohn, P.~M.$\,$: {\it Free rings and their relations},
second edition, London Math.\ Soc.\ Monograph {\bf 19}, London (1985)}
\lit{[Cu--Pi]} {Cutkosky, S.~D.\ -- Piltant, O.$\,$: {\it Ramification
of valuations}, Adv.\ in Math.\ {\bf 183} (2004), 1--79}
\lit{[Del]} {Delon, F.$\,$: {\it Quelques propri\'et\'es des corps
valu\'es en th\'eo\-ries des mod\`eles}, Th\`ese Paris~VII (1981)}
\lit{[Deu]} {Deuring, M.$\,$: {\it Lectures on the theory of algebraic
functions in one variable}, Springer LNM {\bf 314} (1972)}
\lit{[Dr--Ku]} {van den Dries, L.\ -- Kuhlmann, F.-V.: {\it Images of
additive polynomials in $\F_q((t))$ have the optimal approximation
property}, Can.\ Math.\ Bulletin {\bf 45} (2002), 71--79}
\lit{[En]} {Endler, O.$\,$: {\it Valuation theory}, Springer, Berlin
(1972)}
\lit{[Ep]} {Epp, Helmut H.\ P.$\,$: {\it Eliminating Wild Ramification},
Inventiones Math.\ {\bf 19} (1973), 235--249}
\lit{[Er1]} {Ershov, Yu.\ L.$\,$: {\it On the elementary theory of
maximal valued fields III} (in Russian), Algebra i Logika {\bf 6}:3
(1967), 31--38}
\lit{[Er2]} {Ershov, Yu.\ L.$\,$: {\it Multi-valued fields}, Kluwer,
New York (2001)}
\lit{[F--Jr]} {Fried, M.\ -- Jarden, M.$\,$: Field Arithmetic,
Springer, Berlin (1986)}
\lit{[Ge]} {Gentile, E.~R.$\,$: {\it On rings with a one-sided field
of quotients}, Proc.\ AMS {\bf 11} (1960), 380--384}
\lit{[Go]} {Goss, D.$\,$: {\it Basic Structures of Function Field
Arithmetic }, Springer, Berlin (1998)}
\lit{[Gra]} {Gravett, K.~A.~H.$\,$: {\it Note on a result of Krull},
Cambridge Philos.\ Soc.\ Proc.\ {\bf 52} (1956), 379}
\lit{[Gre]} {Greenberg, M.~J.$\,$: {\it Rational points in henselian
discrete valuation rings}, Publ.\ Math.\ I.H.E.S.\ {\bf 31} (1967),
59--64}
\lit{[H]} {Huppert, B.$\,$: {\it Endliche Gruppen I}, Springer, Berlin
(1967)}
\lit{[J]} {Jacobson, N.$\,$: {\it Lectures in abstract algebra, III.\
Theory of fields and Galois theory}, Springer Graduate Texts in Math.,
New York (1964)}
\lit{[Ka1]} {Kaplansky, I.$\,$: {\it Maximal fields with valuations I},
Duke Math.\ J.\ {\bf 9} (1942), 303--321}
\lit{[Ka2]} {Kaplansky, I.$\,$: {\it Selected papers and other
writings}, Springer, New York (1995)}
\lit{[Kr]} {Krull, W.$\,$: {\it Allgemeine Bewertungstheorie},
J.\ reine angew.\ Math.\ {\bf 167} (1931), 160--196}
\lit{[Ku1]} {Kuhlmann, F.-V.: {\it Henselian function fields and tame
fields}, extended version of Ph.D.\ thesis, Heidelberg
(1990)}
\lit{[Ku2]} {Kuhlmann, F.-V.: {\it Valuation theory of fields, abelian
groups and modules}, Habilitation thesis, Heidelberg (1995)}
\lit{[Ku3]} {Kuhlmann, F.--V.$\,$: {\it Valuation theoretic and model
theoretic aspects of local uniformization},
in: Resolution of Singularities --- A
Research Textbook in Tribute to Oscar Zariski. Herwig Hauser, Joseph
Lipman, Frans Oort, Adolfo Quiros (eds.), Progress in Mathematics Vol.\
{\bf 181}, Birkh\"auser Verlag Basel (2000), 381-456}
\lit{[Ku4]} {Kuhlmann, F.-V.: {\it Elementary properties of power series
fields over finite fields}, J.\ Symb.\ Logic {\bf 66} (2001), 771--791}
\lit{[Ku5]} {Kuhlmann, F.-V.$\,$: {\it A correction to Epp's paper
``Elimination of wild ramification''}, Inventiones Math.\ {\bf 153}
(2003), 679--681}
\lit{[Ku6]} {Kuhlmann, F.-V.: {\it Value groups, residue fields and bad
places of rational function fields}, Trans.\ Amer.\ Math.\ Soc.\ {\bf
356} (2004), 4559--4600}
\lit{[Ku7]} {Kuhlmann, F.--V.$\,$: {\it Places of algebraic function
fields in arbitrary characteristic}, Advances in Math.\ {\bf 188}
(2004), 399--424}
\lit{[Ku8]} {Kuhlmann, F.--V.$\,$: {\it A classification of Artin
Schreier defect extensions and a characterization of defectless fields},
submitted}
\lit{[Ku9]} {Kuhlmann, F.--V.$\,$: {\it Elimination of Ramification
I: The Generalized Stability Theorem}, submitted}
\lit{[Ku10]} {Kuhlmann, F.--V.$\,$: {\it Elimination of Ramification
II: Henselian Rationality}, in preparation}
\lit{[Ku11]} {Kuhlmann, F.--V.$\,$: {\it The model theory of tame valued
fields}, in preparation}
\lit{[Ku12]} {Kuhlmann, F.--V.$\,$: Book in preparation. Preliminary
versions of several chapters available at:\n
{\tt http://math.usask.ca/$\,\tilde{ }\,$fvk/Fvkbook.htm}}
\lit{[Ku--Pa--Ro]} {Kuhlmann, F.-V.\ -- Pank, M.\ -- Roquette, P.$\,$:
{\it Immediate and purely wild extensions of valued fields},
Manuscripta Math.\ {\bf 55} (1986), 39--67}
\lit{[L]} {Lang, S.$\,$: {\it Algebra},
Addison-Wesley, New York (1965)}
\lit{[N]} {Neukirch, J.$\,$: {\it Algebraic number theory},
Springer, Berlin (1999)}
\lit{[O1]} {Ore, O.$\,$: {\it Theory of non-commutative polynomials},
Ann.\ Math.\ {\bf 34} (1933), 480--508}
\lit{[O2]} {Ore, O.$\,$: {\it On a special class of polynomials},
Trans.\ Amer.\ Math.\ Soc.\ {\bf 35} (1933), 559--584}
\lit{[Ph--Za]} {Pheidas, T.\ --- Zahidi, K.$\,$: {\it Elimination
theory for addition and the Frobenius map in polynomial rings},
J.\ Symb.\ Logic {\bf 69} (2004), 1006--1026}
\lit{[Pop]} {Pop, F.$\,$: {\it \"Uber die Structur der rein wilden
Erweiterungen eines K\"orpers}, manuscript, Heidelbeg (1987)}
\lit{[R--Za]} {Ribes, L.\ -- Zalesskii, P.$\,$: {\it Profinite Groups},
Springer, Berlin (2000)}
\lit{[Ri]} {Ribenboim, P.$\,$: {\it Th\'eorie des valuations}, Les
Presses de l'Uni\-versit\'e de Mont\-r\'eal, Mont\-r\'eal, 2nd ed.\
(1968)}
\lit{[Roh]} {Rohwer, T.$\,$: {\it Valued difference fields as modules
over twisted polynomial rings}, Ph.D.\ thesis, Urbana (2003).
Available at:\n
{\tt http://math.usask.ca/fvk/theses.htm}}
\lit{[V]} {V\'amos, P.$\,$: {\it Kaplansky fields and $p$-algebraically
closed fields}, Comm.\ Alg.\ {\bf 27} (1999), 629--643}
\lit{[Wa1]} {Warner, S.$\,$: {\it Nonuniqueness of immediate maximal
extensions of a valuation}, Math.\ Scand.\ {\bf 56} (1985), 191--202}
\lit{[Wa2]} {Warner, S.$\,$: {\it Topological fields}, Mathematics
studies {\bf 157}, North Holland, Amsterdam (1989)}
\lit{[Wh1]} {Whaples, G.$\,$: {\it Additive polynomials}, Duke
Math.\ Journ.\ {\bf 21} (1954), 55--65}
\lit{[Wh2]} {Whaples, G.$\,$: {\it Galois cohomology of additive
polynomials and n-th power mappings of fields}, Duke Math.\
Journ.\ {\bf 24} (1957), 143--150}
\lit{[Zi]} {Ziegler, M.$\,$: Die elementare Theorie der henselschen
K\"orper, {\sl Inaugural Dissertation}, K\"oln (1972)}
\end{reference}
\adresse
\end{document}